\documentclass[11pt]{amsart}
\usepackage[mathscr]{eucal}
\usepackage{palatino, mathpazo, amsfonts, mathrsfs, amscd}
\usepackage[all]{xy}
\usepackage{url}
\usepackage{amssymb, amsmath, amsthm}
\usepackage{stackrel}
\usepackage{bm}
\usepackage{hyperref}

\usepackage{tikz-cd}
\usepackage{tikz}
\usetikzlibrary{arrows}

\makeatletter
\DeclareFontFamily{OMX}{MnSymbolE}{}
\DeclareSymbolFont{MnLargeSymbols}{OMX}{MnSymbolE}{m}{n}
\SetSymbolFont{MnLargeSymbols}{bold}{OMX}{MnSymbolE}{b}{n}
\DeclareFontShape{OMX}{MnSymbolE}{m}{n}{
	<-6>  MnSymbolE5
	<6-7>  MnSymbolE6
	<7-8>  MnSymbolE7
	<8-9>  MnSymbolE8
	<9-10> MnSymbolE9
	<10-12> MnSymbolE10
	<12->   MnSymbolE12
}{}
\DeclareFontShape{OMX}{MnSymbolE}{b}{n}{
	<-6>  MnSymbolE-Bold5
	<6-7>  MnSymbolE-Bold6
	<7-8>  MnSymbolE-Bold7
	<8-9>  MnSymbolE-Bold8
	<9-10> MnSymbolE-Bold9
	<10-12> MnSymbolE-Bold10
	<12->   MnSymbolE-Bold12
}{}

\let\llangle\@undefined
\let\rrangle\@undefined
\DeclareMathDelimiter{\llangle}{\mathopen}%
{MnLargeSymbols}{'164}{MnLargeSymbols}{'164}
\DeclareMathDelimiter{\rrangle}{\mathclose}%
{MnLargeSymbols}{'171}{MnLargeSymbols}{'171}
\makeatother

\newtheorem{theorem}{Theorem}[section]
\newtheorem{lemma}[theorem]{Lemma}

\newtheorem{proposition}[theorem]{Proposition}
\newtheorem{corollary}[theorem]{Corollary}

\theoremstyle{remark}
\newtheorem{remark}[theorem]{Remark}
\newtheorem{example}[theorem]{Example}
\newtheorem{definition}[theorem]{Definition}
\newtheorem{notation}[theorem]{Notation}

\numberwithin{equation}{subsection}
\usepackage{todonotes}
\newtheorem{assumption}[theorem]{Assumption}
\newtheorem{condition}[theorem]{Condition}

\makeatletter
\def\imod#1{\allowbreak\mkern10mu({\operator@font mod}\,\,#1)}
\makeatother

\newcommand{\sslash}{\mathbin{/\mkern-6mu/}}

\newcommand{\op}[1]{\operatorname{#1}}

\newcommand{\ol}{\overline}
\newcommand{\wt}{\widetilde}
\newcommand{\wh}{\widehat}

\newcommand{\cc}[1]{\mathcal{#1}}  

\newcommand{\CC}{\mathbb{C}}
\newcommand{\ZZ}{\mathbb{Z}}
\newcommand{\LL}{\mathbb{L}}
\newcommand{\KK}{\mathbb{K}}
\newcommand{\PP}{\mathbb{P}}
\newcommand{\QQ}{\mathbb{Q}}

\newcommand{\RR}{\mathbb{R}}

\newcommand{\sMbar}{\overline{\mathscr{M}}}

\newcommand{\bq}{\bm{q}}
\newcommand{\bee}{\mathbf{e}}
\newcommand{\bff}{\mathbf{f}}

\newcommand{\bt}{\bm{t}}

\newcommand{\bs}{\bm{s}}

\newcommand{\cX}{\cc X}

\newcommand{\cY}{\cc Y}
\newcommand{\cT}{\cc T}
\newcommand{\cC}{\cc C}
\newcommand{\cS}{\cc S}
\newcommand{\cM}{\cc M}
\newcommand{\cU}{\cc U}
\newcommand{\cV}{\cc V}

\newcommand{\cZ}{\cc Z}

\newcommand{\sL}{\mathscr L}
\newcommand{\sA}{{\mathscr{A}}}
\newcommand{\bnab}{\boldsymbol{\nabla}}
\newcommand{\bn}{{\bf N}}

\DeclareMathOperator{\End}{End}
\DeclareMathOperator{\amb}{amb}

\DeclareMathOperator{\ch}{ch}

\DeclareMathOperator{\tot}{tot}

\DeclareMathOperator{\Eff}{Eff}
\DeclareMathOperator{\Hom}{Hom}

\newcommand{\set}[1]{\left\{#1\right\}}  

\title{Extremal transitions via quantum Serre duality}
\author[Mi and Shoemaker]{Rongxiao Mi and Mark Shoemaker}
\address{
  \begin{tabular}{l}
	Rongxiao Mi \\
	\hspace{.1in} Harvard University \\
	\hspace{.1in} Center of Mathematical Sciences and Applications \\
	\hspace{.1in} 20 Garden Street, Cambridge, MA, USA, 02138-3602\\
	\hspace{.1in} Email: {\bf rongxiao@cmsa.fas.harvard.edu} \\
~\\
   Mark Shoemaker \\
   \hspace{.1in} Colorado State University \\
      \hspace{.1in} Department of Mathematics \\
   \hspace{.1in} 1874 Campus Delivery, Fort Collins, CO, USA, 80523-1874\\
   \hspace{.1in} Email: {\bf mark.shoemaker@colostate.edu} \\
  \end{tabular}
}

\begin{document}
\maketitle


\begin{abstract}
Two varieties $Z$ and $\widetilde Z$ are said to be related by extremal transition if there exists a degeneration from $Z$ to a singular variety $\overline Z$ and a crepant resolution $\widetilde Z \to \overline Z$.  
In this paper we compare the genus-zero Gromov--Witten theory of toric hypersurfaces related by extremal transitions arising from toric blow-up.
We show that the quantum $D$-module of $\widetilde Z$, after analytic continuation and restriction of a parameter, recovers the quantum $D$-module of $Z$.  
The proof provides a geometric explanation for both the analytic continuation and restriction parameter appearing in the theorem.

%
%
%
%
%

\end{abstract}
\tableofcontents
\section{Introduction}



Two smooth projective varieties $\cZ$ and $\wt \cZ$ are said to be related by an extremal transition if there exists a singular variety $\ol \cZ$ together with a projective degeneration from $\cZ$ to $\ol \cZ$ and crepant resolution $\wt \cZ \to \ol \cZ$.  Topologically the spaces $\cZ$ and $\wt \cZ$ are related by a surgery.

Motivation for studying extremal transitions comes from
birational geometry, 
where they
provide a bridge between different connected components of moduli.  A famous conjecture known as Reid's fantasy \cite{ReidFant} predicts that any pair of smooth Calabi--Yau 3-folds may be connected via a sequence of 
flops and extremal transitions.  Phrased another way, if we consider two varieties as equivalent if they are related by a birational crepant transformation, then the moduli space of Calabi--Yau 3-folds is connected.  This has been verified for a large class of examples \cite{CGGK, ACJM, AKMS}.

Furthermore, Morrison \cite{MTh} showed that extremal transitions are naturally compatible with mirror symmetry.  It was checked in many cases that if $\cZ^m$ is mirror to $\cZ$, then a mirror $\wt \cZ^m$ to $\wt \cZ$ may be constructed by applying a ``dual transition'' to $\cZ^m$, by first contracting and then smoothing.  This relationship was conjectured to hold generally.  

The two conjectures together suggest a tantalizing strategy for understanding mirror symmetry generally.  If one can determine how both the A-model (Gromov--Witten theory) and the B-model (variation of Hodge structures) vary under extremal transitions, then one could-in principle-prove mirror symmetry for an arbitrary Calabi--Yau 3-fold by connecting it via extremal transitions to an example where the mirror theorem is known.  
The first steps towards this ambitious goal were initiated by Li--Ruan \cite{LR}, where the behavior of Gromov--Witten theory was determined for  \emph{small transitions of 3-folds}, a type of extremal transition where the the exceptional locus of $\wt \cZ \to \ol \cZ$ consists of finitely many rational curves. 
For the specific case of conifold transitions between Calabi--Yau 3-folds, the $A$ and $B$ models together were systematically studied by Lee-Lin-Wang in \cite{LLWAB}. 

Beyond the setting of mirror symmetry and 3-folds, the interplay between Gromov--Witten theory and birational geometry is a rich subject in its own right.   One of  most prominent and well-studied examples of this is the \emph{crepant transformation conjecture} (also known as the crepant resolution conjecture) which predicts a deep relationship between the Gromov--Witten invariants of smooth varieties related by crepant birational transformation.  While it is not possible to give a complete account of all of the progress on this subject here, examples of results of this type may be found in \cite{BG, CIT, CR, LLW, CIJ}. 

In contrast to the crepant transformation conjecture, our understanding of the behavior of Gromov--Witten theory under extremal transitions is much less complete.   With the notable exception of small transitions between 3-folds \cite{LR, LLWAB}, there are only a few examples of extremal transitions
where
the Gromov--Witten theory has been studied.
 These examples were studied by Iritani--Xiao in \cite{IX}, and by the first author in \cite{MiCub, Mideg}. Prior to this, it was not clear what kind of statement one should expect for the change of Gromov-Witten theory under general extremal transitions in arbitrary dimensions. 

\subsection{Results}

In this paper we compare the genus-zero Gromov-Witten theory of toric hypersurfaces related by an extremal transition induced from a toric blow-up on the ambient space. This includes a very large family of interesting examples, including non-small extremal transitions in arbitrary dimensions. The setup is roughly as follows.  

Let $\cX$ be a smooth toric Deligne--Mumford stack with projective coarse moduli space and let $\wt \cX \to \cX$ be the blow-up of $\cX$ along a torus-invariant subvariety $\cV$.  Let  $\cZ$ 
be a hypersurface in $\cX$, defined by the vanishing of a general section 
of a semi-ample line bundle.  By degenerating the section appropriately, we obtain a variety $\ol \cZ$ which acquires singularities along $\cV$.  Under certain numerical conditions, the proper transform of $\ol \cZ$ in $\wt \cX$, denoted by $\wt \cZ$, will be a crepant resolution of $\ol \cZ$.  In this case $\cZ$ and $\wt \cZ$ are related by extremal transition through the singular variety $\ol \cZ$ (see Section~\ref{s:ts} for details).  In the case where $\cZ$ and $\wt \cZ$ are Calabi--Yau, this was one of the cases considered by Morrison in formulating his conjectural connection to mirror symmetry \cite[Section~3.2]{MTh}.

Our results are formulated in terms of the \emph{quantum $D$-module} $QDM(\cZ)$,  a collection of data consisting of:
\begin{itemize}
\item the Dubrovin connection $\nabla^{\cZ}$, a flat connection lying over the extended K\"ahler moduli and encoding the genus-zero Gromov--Witten theory of $\cZ$;
\item a pairing $S^\cZ$, flat with respect to $\nabla^{\cZ}$.
\end{itemize}
More precisely, we consider the \emph{ambient} quantum $D$-module of $\cZ$, $QDM_{\op{amb}}(\cZ)$, defined by restricting the state space $H^*(\cZ)$ to those insertions pulled back from $\cX$.  An additional important piece of data is the $\wh \Gamma$-integral structure defined by Iritani.  This is a lattice of $\nabla^\cZ$-flat sections  determined by the $K$-theory $K^0(\cZ)$.

The main theorem may be paraphrased as follows:
\begin{theorem}[Theorem~\ref{t:mt}]\label{t:mtpara}
The Dubrovin connection $\nabla^{\wt \cZ}$ is analytic along a specified direction $y_{r+1}$, and may be analytically continued to a neighborhood of $y_{r+1} = \infty$.  The monodromy invariant part around $y_{r+1}  = \infty$, when restricted to $y_{r+1}  = \infty$, contains a subquotient which is gauge-equivalent to $\nabla^\cZ$.  This equivalence is compatible with the pairing and integral structures on $QDM_{\op{amb}}(\cZ)$ and $QDM_{\op{amb}}(\wt \cZ)$.
\end{theorem}
The form of the correspondence given above is a slight modification of that conjectured by the first author in \cite{Mideg}.  A similar formulation also appeared in \cite{IX}.

It is interesting to note that the theorem described above consists of two distinct steps: 
(1) \textbf{analytic continuation} of $\nabla^{\wt \cZ}$; and (2) \textbf{restriction} of flat sections to $y_{r+1} = \infty$ (restricting $\tilde y_{r+1} = 1/y_{r+1} $ to zero).  
The appearance of analytic continuation is reminiscent of the crepant transformation conjecture, which can be formulated similarly \cite{CIT, CIJ}, but without a restriction of variables.  Our proof yields a geometric explanation for both of the above steps.

To our knowledge this is the first result on the Gromov--Witten theory of extremal transitions that includes varieties of arbitrary dimension and includes a large family of non-small extremal transitions.  Furthermore, a (somewhat surprising) upshot of our proof is that the integral lattices may also be compared. 
We expect the techniques developed here to apply more broadly.
We plan to explore this in future work, with the hope of better understanding mirror symmetry for a large class of varieties. 
\subsection{Strategy of proof}

At the heart of our proof is the use of \emph{quantum Serre duality} to compare quantum $D$-modules.  Let $D$ and $\wt D$ denote the divisors on $\cX$ and $\wt \cX$ which define the hypersurfaces $\cZ$ and $\wt \cZ$ respectively.  Define
\[ \cT := \tot( \cc O_{\cX} (- D)) \text{ and }  \wt \cT := \tot( \cc O_{\wt \cX}(-\wt D)).\]
Originally proven by Givental \cite{G1}, quantum Serre duality is the name given to a correspondence between the genus-zero Gromov--Witten theory of $\cZ$ and $\cT$ (resp. $\wt \cZ$ and $\wt \cT$).
This was reformulated by the second author in \cite{ShNar}
as an isomorphism between the ambient quantum $D$-module of $\cZ$ and the \emph{narrow} quantum $D$-module of $\cT$ (resp. $\wt \cZ$ and $\wt \cT$).  The narrow quantum $D$-module is defined by restricting the state space $H^*(\cT)$ to those insertions which can be represented by classes of compact support.

It therefore suffices to compare the (narrow) quantum $D$-modules of $\cT $ and $\wt \cT$.  The benefit of this perspective is the existence of an intermediate toric variety, denoted $\ol \cT$, which relates to both $\cT$ and $\wt \cT$.  The toric variety $\ol \cT$ is obtained from $\wt \cT$ by a flop, and is simultaneously a partial compactification of $\cT$, compactifying the fibers of $\tot(\cc O_{\cX} (- D)|_{\cV}) \subset \cT$.

Figure~\ref{f3} provides a picture of the fans in the case of $\cX = \PP^2$, $D = -K_\cX$ and $\cV$ a point.  In this picture the bottom vertex is the origin and each line out of the origin is a primitive ray vector.

 \begin{figure}[h]\includegraphics[width=\textwidth]{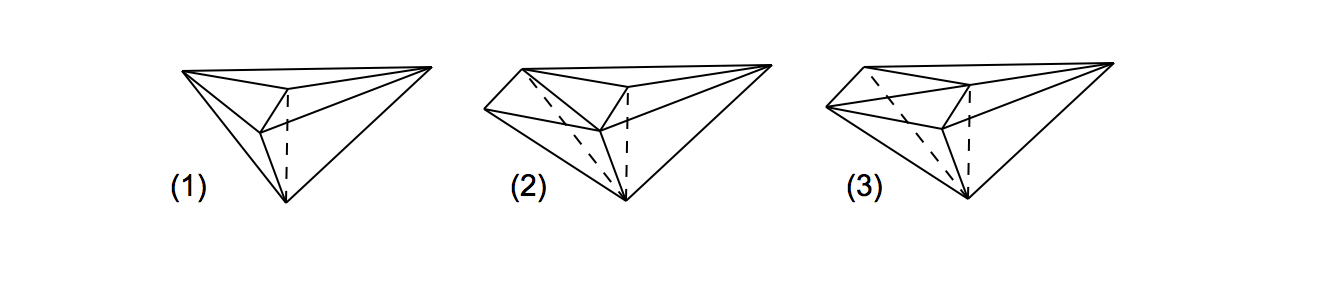}
\caption{Fans representing the toric varieties (1) $\cT$, (2) $\ol \cT$, and (3) $\wt \cT$.}
\label{f3}
  \end{figure}

The proof then proceeds in two steps, each of independent interest.  First, using the work of \cite{CIJ}, we prove a narrow version of the crepant transformation conjecture for non-compact toric varieties.  In particular this applies to $\wt \cT$ and $\ol \cT$.
\begin{theorem}[Theorem~\ref{t:nctc}]\label{t:ctcpara}
Given toric varieties $\cY_-$ and $\cY_+$ related by a crepant variation of GIT across a codimension-one wall, the narrow quantum $D$-modules are gauge equivalent after analytic continuation along a specified path.  This equivalence preserves the narrow pairing and identifies the integral lattices.
\end{theorem}

Second, we use the specific geometry of the partial compactification $$i: \cT \hookrightarrow \ol \cT$$ to compare their respective Gromov--Witten invariants.
\begin{theorem}[Corollary~\ref{c:monod}]\label{t:restpara}
The monodromy invariant part of $QDM_{\op{nar}}(\ol \cT)$ around $q' = 0$ (where $q'$ is a coordinate corresponding to the divisor $\ol \cT \setminus \cT$), when restricted to $q' = 0$, contains a submodule which maps surjectively to $QDM_{\op{nar}}(\cT)$ via $i^*$.  This map is compatible with the flat pairing, and identifies the integral lattice generated by $K^0_\cX(\ol \cT)$ with that generated by $K^0_\cX(\cT)$.
\end{theorem}
Theorem~\ref{t:mtpara} then proceeds by combining Theorems~\ref{t:ctcpara} and~\ref{t:restpara} with quantum Serre duality.  Schematically we have:
\begin{equation}\label{eseq} \begin{tikzcd}
 \ar[d, leftrightarrow, "\text{QSD}"]  QDM_{\op{nar}}(\cT) \ar[r, leftarrow, "\text{restrict}"] & QDM_{\op{nar}}(\ol \cT) 
\ar[r, leftrightarrow, "\text{CTC}"] & QDM_{\op{nar}}(\wt \cT) \ar[d, leftrightarrow, "\text{QSD}"] 
\\
QDM_{\op{amb}}(\cZ) & & QDM_{\op{amb}}(\wt \cZ).
\end{tikzcd}
\end{equation}
From this perspective, the variety $\ol \cT$ plays a crucial role in understanding the relationship between the quantum $D$-modules of $\cZ$ and $\wt \cZ$.
The appearance of analytic continuation in Theorem~\ref{t:mtpara} is explained by 
the crepant transformation between $\wt \cT$ and $\ol \cT$, and
the restriction of parameters is due to the partial compactification $\cT \hookrightarrow \ol \cT$.

\subsection{Connections to other results}

The behavior of Gromov--Witten theory under small transitions among 3-folds was first studied in \cite{LR}.   In \cite{LLWAB}, it was shown that the $A$ and $B$ model of $\cZ$ considered \emph{together} are determined by the $A$ and $B$ models of $\wt \cZ$ for conifold transitions between Calabi-Yau 3-folds.  Beyond 3-folds, particular cases of small transitions arising as toric degenerations were considered in \cite{IX}, where a similar statement to our main theorem was proven.  Examples of non-small extremal transitions were studied by the first author.  These include cases of triple point transitions in \cite{MiCub} and degree-4 Type II transitions in \cite{Mideg}.  In the latter paper, a general conjecture was formulated using the language of quantum D-modules.

The Calabi--Yau case of the toric hypersurface transitions we consider have appeared before, for instance in showing that for large classes of known examples, the "web" of Calabi--Yau 3-folds is connected \cite{CGGK, ACJM, AKMS}, an important step towards verifying Reid's fantasy.

The intermediate space $\ol \cT$ has appeared previously in the context of Landau--Ginzburg models.  
After adding an appropriate superpotential $\ol w: \ol \cT \to \CC$, the corresponding LG model is known as an \emph{exoflop}, and was studied in \cite{Asp}.  In the context of derived categories, the category of matrix factorizations on the LG model $\ol w: \ol \cT \to \CC$ was used to prove an equivalence of categories between the derived categories arising from different Berglund--H\"ubsch--Krawitz mirrors \cite{FK}.  
 
\subsection{Acknowledgements}
The authors are indebted to Y.-P.~Lee and Y.~Ruan for their mentorship and guidance.  In addition to innumerable other helpful conversations, Y.~Ruan first explained the significance of extremal transitions in Gromov--Witten theory and mirror symmetry.  The role of quantum Serre duality in proving other correspondences was first proposed to M.~S. by Y.-P.~Lee in the context of the LG/CY correspondence. 
M.~S. also thanks H.~Iritani and J.~A.~Cruz Morales for helpful discussions and correspondences, and thanks H.~Iritani for suggesting Proposition~\ref{p:byIri}. R.~M. is supported by a postdoc fellowship from Center for Mathematical Sciences and Applications at Harvard University.
M.~S. was partially supported by NSF grant DMS-1708104.

\section{Gromov--Witten theory preliminaries}
In this section, we will briefly review the basics of Gromov-Witten theory and set notations for the rest of this paper. Our presentation here mainly follows \cite{AGV}.
\subsection{Quantum $D$-module} \label{ss:qdm}
Let $\cX$ be a smooth Deligne-Mumford stack over $\CC$, whose coarse moduli space is projective.  We denote by $I\cX$ the inertia stack of $\cX$, and by $\bar I\cX$ the rigidified inertia stack \cite[Section~3.4]{AGV}.  Recall that points of $I\cX$ are given by pairs $(x, g)$ where $x$ is a point of $\cX$ and $g \in \op{Aut}_\cX(x)$ is an element of the isotropy group of $x$.
We denote the twisted sectors of $I\cX$ by $\cX_\nu$ for $\nu$ ranging over some index set $R$.  Let $$\mathfrak I:I\cX\to I\cX$$ denote the natural involution on components of $I\cX$ which maps $(x,g)$ to $(x, g^{-1})$.
\begin{definition}\label{d:CRcoh}
Define the \emph{Chen--Ruan} (CR) cohomology of $\cX$ to be
\[H^*_{\op{CR}}(\cX) := H^*(I\cX),\]
where cohomology is always taken with coefficients in $\CC$ unless otherwise specified. Define the Chen--Ruan pairing on $H^*_{\op{CR}}(\cX)$ to be
\[\langle\alpha, \beta\rangle^{\cX}:=\int_{I\cX}\mathfrak I^*\beta\cup \alpha.\]
\end{definition}

Let $\Eff = \Eff(\cX) \subset H_2(\cX,\QQ)$ denote the cone of effective curve classes.
For $d \in \Eff$, $g, n \geq 0$, let $\sMbar_{g,n}(\cX,d)$ denote the moduli space of representable degree-$d$ stable maps from genus-$g$ orbi-curves with $n$ markings to the target space $\cX$ \cite[Section~4.3]{AGV}. For $1\leq i\leq n$, we have evaluation maps:
\[\op{ev}_i:\sMbar_{g,n}(\cX,d)\to \bar I\cX.\]
 There is a canonical isomorphism $H^*(I\cX)\cong H^*(\bar I\cX)$, which allows us define 
$\op{ev}_i^*\alpha$ for $\alpha \in H^*_{\op{CR}}(\cX)$. 

\begin{definition}
For $d\in \Eff$ an effective curve class  and $\alpha_1,\cdots,\alpha_n\in H^*_{\op{CR}}(\cX)$, define the Gromov-Witten invariant
\[\langle\alpha_1\psi_1^{b_1}\cdots\alpha_n\psi_n^{b_n}\rangle^{\cX}_{g,n,d}:=\int_{[\sMbar_{g,n}(\cX,d)]^{\op{vir}}} \prod_{i=1}^n \op{ev}_i^*\alpha_i\cup \psi_i^{b_i}, \]
where $\psi_i$ is the first Chern class of the $i-$th cotangent line bundle over $\sMbar_{g,n}(\cX,d)$, and $[\sMbar_{g,n}(\cX,d)]^{\op{vir}}$ is the virtual fundamental class defined as in \cite{BF} and \cite[Section~4.4]{AGV}. 
\end{definition}
\begin{definition} 
We introduce the double bracket notation for generating series of genus-zero Gromov-Witten invariants. For $\alpha_1,\cdots,\alpha_n\in H^*_{\op{CR}}(\cX)$, and $\bt$ a point in  $ H^*_{\op{CR}}(\cX)$,
\[\llangle\alpha_1\psi_1^{b_1},\cdots,\alpha_n\psi_n^{b_n}\rrangle^\cX(Q, \bt):=\sum_{d \in \Eff}Q^d \sum_{k\geqslant 0}\frac1{k!}\langle\alpha_1\psi_1^{b_1}\cdots\alpha_n\psi_n^{b_n},
\bt,\cdots,\bt\rangle^\cX_{0,n+k,d}.\]
Here the sum is over all terms in the stable range: $(n+k, d) \neq (0,0), (1,0), (2,0)$ and the variables $Q^d$ lie in an appropriate \emph{Novikov Ring} \cite[III 5.2.1]{Ma}.
If $\bt$ is a formal parameter then the restriction to $Q=1$ is well-defined.  By a slight abuse of notation we denote 
\[\llangle\alpha_1\psi_1^{b_1},\cdots,\alpha_n\psi_n^{b_n}\rrangle^\cX(\bt) := \llangle\alpha_1\psi_1^{b_1},\cdots,\alpha_n\psi_n^{b_n}\rrangle^\cX(Q, \bt)|_{Q=1}.\]
\end{definition}
\begin{definition}
Let $\langle-, -\rangle^\cX$ denote the Poincare pairing of $\cX$.  For $\alpha,\beta\in H^*_{\op{CR}}(\cX)$, the quantum product $\alpha\star_{\bt} \beta$ is defined by the equation
\[\langle\alpha\star_{\bt}\beta,\gamma\rangle^\cX=\llangle\alpha,\beta,\gamma\rrangle^\cX(Q, \bt),\]
for all $\gamma\in H^*_{\op{CR}}(\cX).$

The Chen--Ruan product $\alpha \cdot \beta$ is defined by equation 
\[\langle\alpha \cdot \beta,\gamma\rangle^\cX=\llangle\alpha,\beta,\gamma\rrangle^\cX(0, 0),\]
\end{definition}
Unless otherwise specified, when multiplying classes in $H^*_{CR}(\cX)$ we always use the Chen--Ruan product.

\begin{notation}\label{n:basis}
Fix a basis $\{\phi_i\}_{i\in I}$ of $H^*_{\op{CR}}(\cX)$ in such a way that there exists a partition $I ' \coprod I'' \subset I$ such that $\{\phi_i\}_{i\in I''}$ is a basis for $H^2(\cX)$ and $\{\phi_i\}_{i\in I'}$ forms a basis for $H^{\neq 2}(\cX) \oplus \bigoplus_{\nu \neq \op{id}} H^*(\cX_\nu).$
Denote the dual basis by $\{\phi^i\}_{i\in I}$. For any $\bt\in H^*_{\op{CR}}(\cX)$, we write $\bt=\sum_{i\in I} t^i\phi_i$. We denote $q_i:=e^{t^i}$ for each $i\in I''$.  We will make use of the rings $\CC[[\bt]] = \cc[[t^i]]_{i \in I}$ and $\CC[[\bq, \bt']] = \CC[[q_i]]_{i\in I ''}\otimes \CC[[t^i]]_{i\in I '}$.
\end{notation}
By the divisor equation, 
\begin{equation}\label{e:div} \llangle\alpha_1,\cdots,\alpha_n\rrangle^\cX(Q, \bt) = \llangle\alpha_1,\cdots,\alpha_n\rrangle^\cX(Q \cdot \bq, \bt'),\end{equation}
where
$$\bq^d := \prod_{i \in I '' } q_i^{\langle \phi_i, d\rangle^\cX}.$$
Consequently, the quantum product $\alpha \star_{\bt} \beta$  lies in $H^*_{\op{CR}}(\cX)[[Q, \bq, \bt ']]$, and the restriction to $Q=1$ is well-defined.
\begin{definition}
The Dubrovin connection for $\cX$ is given by the collection of operators $\nabla^{\cX}_i\in \End(H^*_{\op{CR}}(\cX))[[Q, \bq, \bt ']][z,z^{-1}]$  defined by
\[\nabla^{\cX}_i=\partial_i+\frac1z\phi_i\star_{\bt}(-)\]
for $i \in I$.

One can extend the definition of $\nabla^\cX$ to the $z$-direction as well.
Define the Euler vector field
\[ \mathfrak E := \partial_{\rho(\cX)}+ \sum_{i \in I} \left(1 - \frac{1}{2} \deg \phi_i\right) t^i \partial_i\]
where $\rho(\cX):= c_1(\cT\cX) \subset H^2(\cX)$.

Define the grading operator $\op{Gr}$ by 
\[\op{Gr}(\alpha) := \frac{\deg \alpha}{2} \alpha\] for $\alpha$ in $H^*_{\op{CR}}(\cX)$.
We define
\[ \nabla_{z\partial_z}^\cX := z\partial_z - \frac{1}{z} \mathfrak E \bullet^\cX_{\bt} +  \op{Gr}.\]
\end{definition}
This meremorphic connection is flat by standard arguments using the WDVV equation and homogeneity \cite{CK}.

\begin{definition}[Definition~3.1 of \cite{Iri3}]
The fundamental solution operator $$L^{\cX}(Q, \bt,z)\in \End(H^*_{\op{CR}}(\cX))\otimes \CC[[\bt]][z^{-1}]$$ is defined as
\[L^{\cX}(Q,\bt,z)\alpha:=\alpha +\sum_{i\in I}\left\llangle\frac{\alpha}{-z-\psi_1},\phi^i\right\rrangle^\cX(\bt)\phi_i,\]
where the denominator  should be interpreted as a power series expansion in $1/z$.
This may be expressed alternatively as
\begin{equation}\label{e:Lalt} 
\alpha + \sum_{d \in \Eff}Q^d \sum_{n\geqslant 0}\frac1{n!}
\mathfrak I_* {\op{ev}_{n+2}}_*
\left(\frac{\op{ev}_1^*(\alpha)}{-z-\psi_1} 
\prod_{k=2}^{n+1} \op{ev}_k^*(\bt) \cap [\sMbar_{g,n+2}(\cX,d)]^{\op{vir}}
\right)
\end{equation}
\end{definition}
By the divisor equation, 
\begin{equation}\label{e:Ldivred} L^{\cX}(Q,\bt,z) \alpha= L^\cX(Q\cdot \bq, \bt ', z) \prod_{i \in I''} e^{t^i\phi_i/z} \alpha.\end{equation}
We denote the restriction to $Q=1$ by
$$L^{\cX}(\bt,z) := L^{\cX}(Q,\bt,z)|_{Q=1}.$$

\begin{remark}
It will be convenient for many of the results that follow to restrict the Novikov variable to one.  We do this without further comment in what follows, with the notable exception of Section~\ref{s:Dcomp}, where the Novikov variable plays a crucial role.  By \eqref{e:div}, there is no loss of information in this restriction.
\end{remark}

	\begin{definition}
		Define a $z$-sesquilinear pairing $S^\cX(-,-)$ on $H^*_{\op{CR}}(\cX)[[\bq, \bt ']]((z^{-1}))$ by
		\[S^{\cX}(u(z),v(z))=(2\pi iz)^{\dim \cX}\langle u(-z),v(z)\rangle^\cX.\]
	\end{definition}

\begin{proposition}[Proposition~2.4 of \cite{Iri} and Remark~3.2 of \cite{Iri3}]\label{p:flatconn}
The quantum connection $\nabla^\cX$ is flat.
The operator $$L^{\cX}(\bt,z)z^{- \op{Gr}}z^{\rho(\cX)} \in \End(H^*_{\op{CR}}(\cX))\otimes \CC[[\bt]][\log z]((z^{-1/k}))$$ is a fundamental solution for the quantum connection:
\begin{equation} \nabla^\cX_i \left(L^\cX(\bt, z) z^{- \op{Gr}}z^{\rho(\cX)} \alpha\right) = \nabla^\cX_{z \partial z} \left(L^\cX(\bt, z) z^{- \op{Gr}} z^{\rho(\cX)}\alpha\right) = 0 \end{equation}  for  $i \in I$ and $\alpha \in H^*_{\op{CR}}(\cX)$.
The pairing $S^\cX$ is flat with respect to $\nabla^{\cX}$.
 For $\alpha,\beta\in H^*_{\op{CR}}(\cX)$, we have $$\langle L^\cX(\bt, z) z^{- \op{Gr}}z^{\rho(\cX)} \alpha, L^\cX(\bt, z) z^{- \op{Gr}}z^{\rho(\cX)} \beta \rangle^\cX=\langle \alpha,\beta\rangle^\cX.$$ \\
\end{proposition}

\begin{definition}
The quantum $D$-module  is defined to be the data $$QDM(\cX) := (H^*_{\op{CR}}(\cX)[[\bq, \bt ']][z,z^{-1}], S^\cX, \nabla^{\cX}).$$
\end{definition}

\subsection{Narrow quantum $D$-module}\label{ss:nqdm}
Many of the definitions above can be naturally extended to a non-proper target $\cY$, provided the evaluation maps are proper. In this scenario, we will introduce an analogous notion of quantum $D$-module, namely, narrow quantum $D$-module.  See \cite{ShNar} for details of this construction.

Let $\cY$ be a smooth Deligne-Mumford stack with quasi-projective coarse moduli space.  In this section we assume that all evaluation maps
\[\op{ev}_i:\sMbar_{g,n}(\cY,d)\to \bar I \cY\]
are proper. 
\begin{definition}
Let $H^*_{\op{CR, c}}(\cY) = H^*_{\op{c}}(I\cY)$ be the Chen-Ruan cohomology with compact support. There is natural forgetful morphism $$\varphi: H^*_{\op{CR, c}}(\cY)\to H^*_{\op{CR}}(\cY).$$ Define the narrow (Chen--Ruan) cohomology of $\cY$, denoted by $H^*_{\op{CR, nar}}(\cY)$, to be the image of $\varphi$ inside $H^*_{\op{CR}}(\cY)$. 
\end{definition}
For $\alpha \in H^*_{\op{CR}}(\cY)$ and $\beta \in H^*_{\op{CR, c}}(\cY)$, define 
\[\langle\alpha, \beta\rangle^{\cY}:=\int_{I\cY}\mathfrak I^*\alpha \cup \beta.\]
There is a non-denegerate pairing on $H^*_{\op{CR, nar}}(\cY)$.  Given $\beta \in H^*_{\op{CR, nar}}(\cY)$, let $\wt \beta \in H^*_{\op{CR, c}}(\cY)$ denote a \emph{lift} of $\beta$, that is an element satisfying $\varphi(\wt \beta) = \beta$.  For $\alpha, \beta \in H^*_{\op{CR, nar}}(\cY)$, we define the pairing
\[\langle\alpha, \beta\rangle^{\cY, \op{nar}}:=\langle\alpha, \wt \beta\rangle^{\cY}.\]
One checks this definition is independent of the choice of lift \cite[Corollary~2.9]{ShNar}.

For  $\alpha_1,\cdots,\alpha_n\in H^*_{\op{CR}}(\cY)\cup H^*_{\op{CR, c}}(\cY)$, the Gromov-Witten invariant
\[\langle\alpha_1\psi^{b_1}\cdots\alpha_n\psi^{b_n}\rangle^{\cY}_{g,n,d}=\int_{[\sMbar_{g,n}(\cY,d)]^{\op{vir}}} \prod_{i=1}^n \op{ev}_i^*\alpha_i\cup \psi^{b_i}\]
is well-defined when at least one of $\alpha_i$ is in $H^*_{\op{CR, c}}(\cY)$.  Again using lifts, the invariant is also well-defined if $\alpha_i,\alpha_j\in H^*_{\op{CR, nar}}(\cY)$ for some $i\neq j$.

Using compactly supported cohomology, most of the definitions in Section~\ref{ss:qdm} go through directly for a non-proper target $\cY$.  For $\alpha, \beta \in H^*_{\op{CR}}(\cY)$ define 
$\alpha \star_{\bt} \beta \in H^*_{\op{CR}}(\cY)[[\bq , \bt']]$ to be the element satisfying 
\[\langle\alpha\star_{\bt} \beta,\gamma \rangle^{\cY}=\llangle\alpha,\beta,\gamma\rrangle^\cY(\bt),\]
for all $\gamma\in H^*_{\op{CR, c}}(\cY)$.
Similarly, for $\alpha \in H^*_{\op{CR}}(\cY)$ and $\beta \in H^*_{\op{CR, c}}(\cY)$, define
$\alpha \star_{\bt} \beta \in H^*_{\op{CR, c}}(\cY)[[\bq , \bt']]$ to be the element satisfying 
\[\langle\alpha\star_{\bt} \beta,\gamma \rangle^{\cY}=\llangle\alpha,\beta,\gamma\rrangle^\cY(\bt),\]
for all $\gamma\in H^*_{\op{CR}}(\cY)$.

The Dubrovin connection $\nabla^\cY$ and fundamental solution $L^{\cY}(\bt,z) z^{- \op{Gr}}z^{\rho(\cY)}$ are defined exactly as in Section~\ref{ss:qdm}. To avoid confusion we denote these  operators by $\nabla^{\cY, \op{c}}$ and $L^{\cY, \op{c}}(\bt,z) z^{- \op{Gr}}z^{\rho(\cY)}$ respectively when they are acting on compactly supported cohomology.

By \cite[Proposition~4.6]{ShNar}, the connection $\nabla^\cY$ and the fundamental solution $L^{\cY}(\bt,z) z^{- \op{Gr}}z^{\rho(\cY)}$ preserve the narrow cohomology.  We denote the restrictions to $H^*_{\op{CR, nar}}(\cY)$ by  $\nabla^{\cY, \op{nar}}$ and  $L^{\cY, \op{nar}}(\bt,z) z^{- \op{Gr}}z^{\rho(\cY)}$ respectively.  

\begin{definition}
Define a $z$-sesquilinear pairing $S^{\cY}(-,-)$ as follows.  For $u(z) \in H^*_{\op{CR}}(\cY)[[\bq, \bt ']][z,z^{-1}]$ and $v(z) \in H^*_{\op{CR, c}}(\cY)[[\bq, \bt ']][z,z^{-1}]$, let
\[S^{\cY}(u(z),v(z))=(2\pi iz)^{\dim \cY}\langle u(-z),v(z)\rangle^{\cY}.\]
For $u(z), v(z) \in H^*_{\op{CR, nar}}(\cY)[[\bq, \bt ']][z,z^{-1}]$, define
\[S^{\cY, \op{nar}}(u(z),v(z))=(2\pi iz)^{\dim \cY}\langle u(-z),v(z)\rangle^{\cY, \op{nar}}.\]
\end{definition}
Analogues of Proposition~\ref{p:flatconn} hold.  
In particular we have the following relation between $\nabla^\cY$ and $\nabla^{\cY, \op{c}}$.
\begin{proposition}[Proposition~4.4 of \cite{ShNar}]\label{p:dualconn}
The connections $\nabla^{\cY}$ and $\nabla^{\cY, \op{c}}$ are dual with respect to the pairing $S^\cY(-,-)$:
\begin{equation}
\partial_i S^\cY(u,v) = S^\cY(\nabla^\cY_i u, v) + S^\cY(u, \nabla^{\cY, \op{c}}_i v).
\end{equation}
Furthermore, for $\alpha \in H^*_{\op{CR}}(\cY)$ and $\beta \in H^*_{\op{CR, c}}(\cY)$,
\begin{equation}
\langle L^\cY(\bt, -z) \alpha, L^{\cY, \op{c}}(\bt, z) \beta \rangle^\cY = \langle \alpha, \beta \rangle.\end{equation}
\end{proposition}
The map $\varphi: H^*_{\op{CR, c}}(\cY) \to H^*_{\op{CR}}(\cY)$ commutes with the quantum connection and the fundamental solution.  It follows that  $\nabla^{\cY, \op{nar}}$ is a flat connection with fundamental solution $L^{\cY, \op{nar}}(\bt,z) z^{- \op{Gr}}z^{\rho(\cY)}$, and the pairing $S^{\cY, \op{nar}}$ is flat with respect to $\nabla^{\cY, \op{nar}}$.  See \cite[Section~4.2]{ShNar} for details.

\begin{definition}
The narrow quantum $D$-module for  $\cY$ is defined to be the data
$$QDM_{\op{nar}}(\cY) := (H^*_{\op{CR, nar}}(\cY)[[\bq, \bt ']][z,z^{-1}], S^{\cY, \op{nar}}, \nabla^{\cY, \op{nar}}).$$
\end{definition}

\subsection{Ambient quantum $D$-module}\label{ss:amb}

In this section we consider a restricted quantum $D$-module of a complete intersection.

Let $\cX$ be a smooth Deligne-Mumford stack with projective coarse moduli space.
\begin{definition}
A vector bundle $E \to \cX$ is convex if, for any stable map $f:\cC\to \cX$ from a genus-zero orbi-curve to $\cX$,
$$H^1(\cC, f^*E) = 0.$$
\end{definition}
Let $E \to \cX$ be a convex vector bundle with a transverse section $s\in \Gamma(\cX, E)$.  Let 
$$k: \cZ \hookrightarrow \cX$$
be the inclusion of the zero locus $\cZ = Z(s)$.

\begin{definition}
Define the \emph{ambient cohomology} of $\cZ$ to be
$$H^*_{\op{CR, amb}}(\cZ):=j^*H^*_{\op{CR}}(\cX).$$
\end{definition}
 We make the following assumption:
\begin{assumption}
The Chen--Ruan pairing on $H^*_{\op{CR, amb}}(\cZ)$ is non-degenerate.
\end{assumption}
The above assumption is equivalent to the statement that $H^*_{\op{CR, amb}}(\cZ)$ decomposes as $\op{ker}(k_*) \oplus \op{im}(k^*)$.  It holds, for instance, if $E$ is a semi-ample line bundle on a toric variety $\cX$ \cite{Mav}, the setting of Section~\ref{s:ts}.

Given $\bar{\bt}\in H^*_{\op{CR, amb}}(\cZ)$, the quantum product $\star_{\bar\bt}$ preserves $H^*_{\op{CR, amb}}(\cZ)$, as do the quantum connection and the fundamental solution \cite[Corollary~2.5]{Iri3}.
We can therefore make the following definition.
\begin{definition}
The ambient quantum $D$-module for $k:\cZ\hookrightarrow \cX$, is defined to be $$QDM_{\amb}(\cZ):=(H^*_{\op{CR, amb}}(\cZ)\otimes \CC[[\bar \bq, \bar \bt ']][z,z^{-1}],S^\cZ, \nabla^{\cZ}),$$
where $\CC[[\bar \bq, \bar \bt ']]$ is defined as in Notation~\ref{n:basis}, but using a basis of $H^*_{\op{CR, amb}}(\cZ)$.
\end{definition}

\subsection{Integral structures}

In \cite{Iri}, Iritani defines an integral lattice of flat sections in $\op{ker}(\nabla^\cX)$.  We review the construction here.

Given $E \to \cX$ a vector bundle, let $E_\nu \to \cX_\nu$ denote the restriction of $E$ to the twisted sector $\cX_\nu$.  Given a point $(x, g_\nu) \in \cX_\nu$, the group $\langle g_\nu \rangle$ acts on $E|_x$.
The vector bundle $E_\nu$ splits as
$$E_\nu = \bigoplus_{0 \leq f <1} E_{\nu, f}$$
where $E_{\nu, f}$ is the Eigenbundle for which the stabilizer $g_\nu$ acts by $e^{2 \pi i f}$.

Let $K^0(\cX)$ denote the Grothendieck group of bounded complexes of vector bundles on $\cX$.  Recall \cite{To, Ts} the definition of the orbifold Chern character:
\begin{align*}
\wt \ch: K^0(\cX) &\to  H^*_{\op{CR}}(\cX) \\
 E &\mapsto \bigoplus_{\nu \in T} \sum_{0 \leq f <1}  e^{2 \pi i f} \ch (E_{\nu, f}).
\end{align*}

\begin{assumption}\label{a:ch}
We assume that the stack $\cX$ has the resolution property, and the map $\wt \ch: K^0(\cX) \to  H^*_{\op{CR}}(\cX)$ is an isomorphism.  
\end{assumption}
Assumption~\ref{a:ch} holds, for instance, if $\cX$ is a smooth toric stack.

\begin{definition}
Define the \emph{Gamma class} $\hat \Gamma: K^0(\cX) \to H^*(I\cX)$ to be the multiplicative class 
$$\wh \Gamma (E) = \bigoplus_{\nu \in T} \prod_{0 \leq f <1} \prod_{i=1}^{\op{rk}(E_{\nu, f})} \Gamma( 1 - f + \rho_{\nu, f, i}),
$$
where $\{\rho_{\nu, f, i}\}_{i=1}^{\op{rk}(E_{\nu, f})}$ are the Chern roots of $E_{\nu, f}$.
Denote by $\wh \Gamma_\cX$ the class $\wh \Gamma(\cT\cX)$.
\end{definition}

Let $$\deg_0: H^*(I\cX) \to H^*(I\cX)$$ denote the operator which multiplies a homogeneous class by its (unshifted) degree.  In \cite{Iri}, Iritani defines a map $\bs^\cX(\bt ,z): K^0(\cX) \to \op{ker}(\nabla^\cX)$.  
\begin{definition}\label{d:intstr}
For $E \in  K^0(\cX) $, define $\bs^\cX(\bt ,z)(E)$ to be
$$ \frac{1}{(2 \pi i)^{\dim \cX}} L^\cX(\bt, z) z^{-\op{Gr}} z^{\rho(\cX)} \left( \wh \Gamma_\cX \cup_{I\cX}  (2 \pi i)^{\op{\deg_0/2}} \mathfrak I^* (\wt \ch(E))\right).$$
\end{definition}

\begin{proposition}[Section~3.2 of \cite{Iri3}]\label{p:pairing} The map $\bs^\cX$ identifies the pairing in $K$-theory with $S^\cX(-,-)$ up to a sign:
$$S^\cX(\bs^\cX(\bt, z)(E), \bs^\cX(\bt, z)(F)) = (-1)^{\dim (\cX)} \chi(E, F).$$
\end{proposition}
\begin{definition}
Define the integral structure of $QDM(\cX)$ to be 
$$\{\bs^{\cX}(\bt, z)(E) | E \in  K^0(\cX)\}$$
\end{definition}
The integral structure of $QDM(\cX)$ forms a lattice in the space of flat sections of $\nabla^\cX$.  

Next we consider the case were $\cY$ is not proper.  We can apply Definition~\ref{d:intstr} without change to obtain an integral lattice in $\ker(\nabla^\cY)$.  With slight modifications we can also define a lattice in $\ker(\nabla^{\cY, \op{c}})$ and $\ker(\nabla^{\cY, \op{nar}})$.

For $\cV$ a closed and proper substack of $\cY$, let $K^0_\cV(\cY)$ denote the Grothendieck group of bounded complexes of vector bundles which are exact off of $\cV$.  Define
$$K^0_{\op{c}}(\cY) := \lim_{\rightarrow} K^0_\cV(\cY)$$
to be the direct limit over all proper substacks $\cV \subset \cY$.  There exists a compactly supported orbifold Chern character $\wt \ch_{\op{c}}: K^0_{\op{c}}(\cY) \to H^*_{\op{CR, c}}(\cY)$ \cite[Definition~8.4]{ShNar}.  Let 
$$\varphi_{K^0}: K^0_{\op{c}}(\cY) \to K^0(\cY)$$
denote the natural map.  By \cite[Proposition~8.5]{ShNar},
\begin{equation}\label{e:phiK}\varphi \circ \wt \ch_{\op{c}} = \wt \ch \circ \varphi_{K^0}.\end{equation}
\begin{assumption}\label{a:ch2}
We assume that the stack $\cY$ has the resolution property and the map $\wt \ch: K^0_{\op{c}}(\cY) \to  H^*_{\op{CR, c}}(\cY)$ is an isomorphism.  
\end{assumption}
Assumption~\ref{a:ch2} holds if $\cY$ is the total space of a vector bundle bundle on a proper toric stack.  

Following Definition~\ref{d:intstr}, we make the following definition.
\begin{definition}\label{d:intstrc}
For $E \in  K^0_{\op{c}}(\cY) $, define $\bs^{\cY, \op{c}}(\bt ,z)(E)$ to be
$$ \frac{1}{(2 \pi i)^{\dim \cY}} L^{\cY, \op{c}}(\bt, z) z^{-\op{Gr}} z^{\rho(\cY)} \left( \wh \Gamma_\cY \cup_{I\cY}  (2 \pi i)^{\op{\deg_0/2}} \mathfrak I^* (\wt \ch_{\op{c}}(E))\right).$$
Define the integral structure of $\ker (\nabla^{\cY,\op{c}})$ to be $$\{\bs^{\cY, \op{c}}(\bt, z)(E) | E \in K^0_{\op{c}}(\cY)\}.$$
\end{definition}
We have the following analogue of Proposition~\ref{p:pairing}:
\begin{proposition}
For $E \in K^0(\cY)$ and $F \in K^0_{\op{c}}(\cY)$, 
$$S^\cY(\bs^\cY(\bt, z)(E), \bs^{\cY, \op{c}}(\bt, z)(F)) = (-1)^{\dim (\cY)} \chi(E, F).$$
\end{proposition}
By \eqref{e:phiK}, $\wt \ch (E)$ is supported in the narrow cohomology for all $E \in \op{im}(\varphi_{K^0}) \subset K^0(\cY)$.  We can therefore define the following.
\begin{definition}
For $E \in \op{im}(\varphi_{K^0}) \subset K^0(\cY)$, define 
$\bs^{\cY, \op{nar}}(\bt ,z)(E)$ to be
$$ \frac{1}{(2 \pi i)^{\dim \cY}} L^{\cY, \op{nar}}(\bt, z) z^{-\op{Gr}} z^{\rho(\cY)} \left( \wh \Gamma_\cY \cup_{I\cY}  (2 \pi i)^{\op{\deg_0/2}} \mathfrak I^* (\wt \ch( E))\right).$$
Define the integral structure of $QDM_{\op{nar}}(\cY)$ to be 
$$\{\bs^{\cY, \op{nar}}(\bt, z)(E) | E \in  \op{im}(\varphi_{K^0}) \subset K^0(\cY)\}.$$
\end{definition}

Let $k: \cZ \to \cX$ be the inclusion of a closed substack in a proper Deligne--Mumford stack $\cX$ as in Section~\ref{ss:amb}.  Again with Assumption~\ref{a:ch}, we can define an integral structure on the ambient quantum $D$-module $QDM_{\op{amb}}(\cZ)$ to be
$$\{\bs^{\cZ, \op{amb}}(\bt, z)(E) | E \in  \op{im}(k^*) \subset K^0(\cZ)\}.$$

\subsection{Quantum Serre duality}\label{s:qqsd}

Consider the setting of Section~\ref{ss:amb}, where $E \to \cX$ is a convex vector bundle, and $k: \cZ \to \cX$ is the zero locus of a transverse section of $E$.  

Consider the total space of the dual bundle:
$$\cT:= \tot(E^\vee).$$
Quantum Serre duality is the name given to a close relationship between the Gromov-Witten theory of $\cT$ and the Gromov-Witten theory of $\cZ$. It was first described mathematically by Givental \cite{G1} and later generalized by Coates--Givental in \cite{CG}.  The correspondence was formulated as a relationship between quantum $D$-modules in \cite{IMM}.  The formulation below using narrow quantum D-modules was given by the second author in \cite{ShNar}.

\begin{definition}\label{d:qsd}
Define the map
$\Delta_\cT:H^*_{\op{CR, nar}}(\cT)[z,z^{-1}]\to H^*_{\op{CR, amb}}(\cZ)[z,z^{-1}]$
as follows.  For $\alpha \in H^*_{\op{CR, nar}}(\cT)$, 
define 
$$\Delta_\cT(\alpha) := k^* \circ \pi^{\op{c}}_*(\tilde \alpha)$$
where $\tilde \alpha \in H^*_{\op{CR, c}}(\cT)$ is a lift of $\alpha$.
This is independent of the choice of lift by \cite[Lemma~6.10]{ShNar}
Define $\bar \Delta_\cT := (2 \pi iz)^{\op{rk}(E)} \Delta_\cT$.
\end{definition}

\begin{definition}
Define the map
$\hat f:H^*_{\op{CR}}(\cT)\to H^*_{\op{CR, nar}}(\cT)[[\bq, \bt']]$
 by
 $$ \hat f(\bt) := \sum_{i \in I}\llangle e(E^\vee),\phi^i\rrangle^\cT (\bt) \phi_i,$$
 where here  $\{\phi_i\}_{i \in I}$ and $\{\phi^i\}_{i \in I}$ are dual bases of $H^*_{\op{CR, nar}}(\cT)$. 
Define $$\bar f(\bt) = \Delta_\cT(\hat f(\bt)) - \pi i c_1(k^* (E^\vee)).$$
 \end{definition}
Note that $\bar f(\bt)$ lies in $H^*_{\op{CR, amb}}(\cZ)[[\bq, \bt']]$.

\begin{theorem}\cite[Theorem~6.14]{ShNar}\label{t:qsd}
The map $\bar \Delta_\cT$ identifies the quantum $D$-module $QDM_{\op{nar}}(\cT)$ with the pullback $\bar f^* QDM_{\op{amb}}(\cZ)$.  Furthermore it is compatible with the integral structures and the functor $k^* \circ \pi_*$, i.e., the following diagram commutes:
\begin{equation}\label{e:comn}
\begin{tikzcd}
K^0(\cT)_{\cX} \ar[r, " k^* \circ \pi_*"] \ar[d, "s^{T, \op{nar}}{(\bt, z)}(-)"] &  k^*(K^0(\cX))  \ar[d, "s^{\cZ, \op{amb}}{(\bar f(\bt), z)}(-)"]\\
QDM_{\op{nar}}(\cT) \ar[r, "\bar \Delta_\cT"] & \bar f^* QDM_{\op{amb}}(\cZ).
\end{tikzcd}
\end{equation}

\begin{remark}
We note that when $\cX$ is a toric stack, the lattice generated by the left vertical map spans the space of solutions in $QDM_{\op{nar}}(\cT)$.
\end{remark}
\end{theorem}

\section{Toric preliminaries}\label{s:tor}

We consider toric varieties constructed as stack quotients via GIT.  As this section is primarily to set notation and recall previous results, the exposition is condensed.  We refer the reader to \cite{CIJ} for further details. 

\subsection{GIT description}\label{ss:GIT}

The initial data consists of 
\begin{itemize}
\item A torus $K \cong (\CC^*)^{r}$;
\item the lattice $\LL = \Hom(\CC^*, K)$ of co-characters of $K$;
\item a set of characters $D_1, \ldots, D_m \in \LL^\vee = \Hom(K, \CC^*)$;
\item a choice of a \emph{stability condition} $\omega \in \LL^\vee \otimes \RR$.
\end{itemize}
Given the above, the map $$(D_1, \ldots, D_m): K \to (\CC^*)^m$$ defines an action of  $K$ on $\CC^m$.  
Given $I \subset \set{1, \ldots, m}$, we denote by $I^{\op{c}}$ the complement of $I$.  

\begin{definition}\label{d:ts} Define $\angle_I \subset \LL^\vee \otimes \RR$ to be the subset 
\[\angle_I := \set{ \sum_{i \in I} a_i D_i | a_i \in \RR_{>0}}.\]
 Define the set of \emph{anticones} (with respect to $\omega$) to be $$\sA_\omega := \set{I \subset \set{1, \ldots, m}| \omega \in \angle_I}.$$  For each $I$, consider the open set \[U_I = (\CC^*)^I \times (\CC)^{I^{\op{c}}}:=  \set{ (x_1, \ldots, x_m) | x_i \neq 0 \text{ for } i \in I} \subset \CC^m\]
 and let $U_\omega$ denote the union
 $U_\omega := \bigcup_{I \in \sA_\omega} U_I.$ 
  We define the toric stack $\cX_\omega$ to be the GIT quotient \[ \cX_\omega := [\CC^m \sslash_\omega K] := [U_\omega / K],\] where brackets denote that we are taking the stack quotient.  
  We will denote the underlying coarse toric variety as
  \[ X_\omega := U_\omega / K.\]
  \end{definition}

  Define the set $S \subset \{1, \ldots, m\}$ to be the collection of $i \in  \{1, \ldots, m\}$ such that $i \in I$ for all $I \in \sA_\omega$.  

\begin{assumption}\label{a:stab}
As in \cite{CIJ}, we assume in what follows that 
\begin{itemize}
\item $\{1, \ldots, m\} \in \sA_\omega$;
\item for each $I \in  \sA_\omega$, the dimension of $\angle_I$ is maximal (equal to $r$).
\end{itemize}
\end{assumption}

The space of stability conditions $\omega \in \LL^\vee \otimes \RR$ has a wall and chamber structure.  Let 
$$C_\omega = \bigcap_{I \in \sA_\omega} \angle_I.$$
Then for any $\omega ' \in C_\omega$, $\cX_\omega = \cX_{\omega '}$.  We call $C_\omega$ the \emph{extended ample cone}.

\subsection{Fan description}

We obtain the more familiar fan description of $\cX_\omega$ from the GIT data as follows.  
Consider the short exact sequence
\begin{equation}\label{e:des} 0 \to \LL \xrightarrow{(D_1, \ldots, D_m)} \ZZ^m \xrightarrow{\beta} \bn \to 0,
\end{equation}
where $\beta$ is defined to be the quotient map of $\ZZ^m$ by $\LL$.  For $1 \leq i \leq m$, let $b_i := \beta(e_i)$ and let 
$\ol b_i$ denote the image of $b_i$ in $\bn \otimes \RR$.  For $I \subseteq {1, \ldots, m}$, let $\sigma_I$ denote the cone generated by $\{\ol b_i\}_{i \in I}$.

\begin{definition}
Define the fan $\Sigma_\omega$ in $\bn \otimes \RR$ to be the collection of all cones $$\Sigma_\omega := \{\sigma_I\}_{I^{\op{c}} \in \sA_\omega}.$$  The collection $(\bn,  \Sigma_\omega, \beta, S)$ is called an \emph{$S$-extended stacky fan}.
\end{definition}
The elements $b_i$ for $i \in S$ are called the \emph{extended vectors}.  It can be checked that for $i \in S$, $\ol b_i$ lies in the support of $\Sigma$.

\subsection{Cohomology}\label{ss:cnc}
Each of the characters $D_i$ for $1 \leq i \leq m$ defines a divisor $u_i \in H^*(\cX_\omega)$.  The class $u_i$ 
may be defined as the first Chern class of the line bundle $\cc L_i \to \cX_\omega$ given by
\begin{equation}\label{e:canline}\cc L_i := [U_\omega \times \CC / K] \to [U_\omega / K],\end{equation}
where the action of $K$ on the last factor of $\CC$ is given by the character $D_i$.  Alternatively $u_i$ is Poincar\'e dual to 
\[\{(x_1, \ldots, x_m) \in U_\omega | x_i = 0\}/K.\]
The cohomology ring of $\cX_\omega$ is then given by the Stanley-Reisner presentation \cite[Lemma~5.1]{BCS}
\begin{equation}\label{e:SR}H^*(\cX_\omega) = \CC[u_1, \ldots, u_m]/(\cc J + \mathfrak J),\end{equation}
where
\begin{align}\label{e:SRrel}
\cc J &= \left\langle \chi - \sum_{i=1}^m \langle \chi, b_i\rangle u_i | \chi \in \bn^\vee \otimes \CC \right \rangle, \\
\mathfrak J &=\left\langle \prod_{i \notin I} u_i | I \notin \sA_\omega \right\rangle. \nonumber
\end{align}
Note that for $i \in S$, $u_i = 0 \in H^*(\cX_\omega)$.  

It will also be convenient to describe the narrow cohomology of $H^*(\cX_\omega)$, which follows from the work of \cite{BHHMS}.  

\begin{definition} Let $\Sigma$ be a fan whose support in $\bn \otimes \RR$ is a rational polyhedral cone.
We say a cone $\sigma \in \Sigma$ is an \emph{interior cone} if the interior $\sigma_I^{\circ}$ of $\sigma_I$ is contained in $|\Sigma|^\circ$, the interior of the support of $\Sigma$.
\end{definition}

\begin{lemma}\label{l:narcoh}
The narrow cohomology $H^*_{\op{nar}}(\cX_\omega)$ is generated as an $H^*(\cX_\omega)$-module by 
$$\left\{ \prod_{i \in I} u_i \;\vline \; \sigma_I \text{ is an interior cone of } \Sigma\right\}.$$
\end{lemma}

\begin{proof}
By \cite[Proposition~2.4]{BHHMS} and the discussion following \cite[Remark~2.5]{BHHMS}, the compactly supported cohomology of $\cX_\omega$ is given by 
$$ \bigoplus_{\sigma_I \in \Sigma, \;\sigma_I^{\circ} \subseteq |\Sigma|^\circ} \CC[u_1, \ldots, u_m] F_I$$
modulo a natural set of relations analogous to those in \eqref{e:SRrel}.  Here $F_I$ is just a formal symbol.  Furthermore, the map $\varphi: H^*_{\op{CR, c}}(\cX_\omega)\to H^*_{\op{CR}}(\cX_\omega)$ is given by
$$F_I \mapsto \prod_{i \in I} u_i.$$
\end{proof}

\subsection{The ample cone}\label{ss:cac}
Equation~\ref{e:SR} implies, in particular, that $H^2(\cX_\omega; \RR) \cong \LL^\vee \otimes \RR / \sum_{j \in S} \RR D_j$.  Using \eqref{e:des}, one can find a canonical isomorphism $\LL^\vee \otimes \RR \cong H^2(\cX_\omega) \oplus  \bigoplus_{j \in S} \RR D_j.$  For each $j \in S$, let $\sigma_{I_j}$ denote the smallest cone of $\Sigma$ which contains $\ol b_j$.  Then $\ol b_j = \sum_{i \in I_j} c_{ij} \ol b_i$ for $c_{ij} \in \QQ_{>0}$.  Define $\xi_j \in \LL \otimes \QQ$ by
\begin{equation}\label{e:xi}
D_i\cdot \xi_j \left\{
\begin{array}{ll}
1 & \text{if } i = j \\
-c_{ij} & \text{if } i \in I_j \\
0 & \text{if } i \notin I_j \cup \{j\}.
\end{array}
\right.
\end{equation}
Then 
\begin{equation}\label{e:h2}
\LL \otimes \RR \cong H_2(\cX_\omega; \RR) \oplus \bigoplus_{j \in S} \RR \xi_j.\end{equation}
Dualizing we see that
\begin{equation}\label{e:split2}\LL^\vee \otimes \RR \cong  \bigcap_{j \in S} \ker( \xi_j) \oplus \bigoplus_{j \in S} \RR D_j,\end{equation}
hence $\bigcap_{j \in S} \ker( \xi_j) \cong H^2(\cX_\omega; \RR)$.  

Denote by $$\theta:  \LL^\vee \otimes \RR \to H^2(\cX; \RR)$$ the quotient by $\langle D_i\rangle_{i \in S}$.  Note that this map satisfies $\theta(D_i) = u_i$.

Under the splitting \eqref{e:split2}, the cone 
\begin{equation}\label{e:ampcone} C_\omega \cong C_\omega ' \times \left( \sum_{i \in S} \RR_{>0} D_i \right) \subset H^2(\cX_\omega; \RR) \times \bigoplus_{j \in S} \RR D_j,\end{equation}
where $C_\omega '$ is the cone of ample divisors.
The Mori cone is then given by
\[ \op{NE}(\cX_\omega) = {C_\omega '}^\vee \subset H_2(\cX_\omega; \RR).\]

\subsection{Chen--Ruan cohomology}
For $\nu \in \LL \otimes \QQ$, define
\[ I_\nu := \{i | D_i \cdot \nu \in \ZZ\}.\]
Define the lattice $\KK_\omega \subset \LL \otimes \QQ$ to be the set
\[ \left\{ \nu \in \LL \otimes \QQ | I_\nu \in \sA_\omega
\right\}.
\]
The set $\KK / \LL$ indexes components of the inertia stack $I\cX_\omega$.  Namely, for $\nu \in \KK$, the corresponding component $\cX_{\omega, \nu}$ of $I\cX_\omega$ is given by 
\[\cX_{\omega, \nu} := [(U_\omega)^{g_\nu} /K],\]
where 
\begin{equation}\label{e:gf}g_\nu := (\op{exp}(2 \pi i D_1 \cdot \nu), \ldots, \op{exp}(2 \pi i D_m \cdot \nu))   \in K.\end{equation}  The component $\cX_{\omega, \nu}$ may be described concretely as 
\[\{[(x_1, \ldots, x_n) \in U_\omega | x_i = 0 \text{ if }  i \notin I_\nu\}/K].\]  In particular, it is also a toric variety constructed as in Section~\ref{ss:GIT}, using $K$ and $\omega$ as before, but restricting to those characters $\{D_i | i \in I_\nu\}.$
One may construct the analogous exact sequence to \eqref{e:des}:
$$0 \to \LL \xrightarrow{(D_i)_{i \in I_\nu}} \ZZ^{|I_\nu|} \xrightarrow{\beta_\nu} \bn_\nu \to 0.$$
Let $\Sigma_{\omega, \nu}$ denote the corresponding fan in $\bn_\nu \otimes \RR$.

The Chen--Ruan cohomology (Definition~\ref{d:CRcoh} below) of $\cX_\omega$ is then given as a vector space by
\[H^*_{\op{CR}}(\cX_\omega) = \bigoplus_{\nu \in \KK/\LL} H^*(\cX_{\omega, \nu}),\]
where $H^*(\cX_{\omega, \nu})$ may be written explicitly by Section~\ref{ss:cnc}.

\subsection{Proper evaluation maps}
Because we do not assume that our toric stack is proper, we include a preliminary lemma to guarantee that the quantum $D$-module (as well as the compactly supported and narrow quantum $D$-module) is well-defined.

\begin{lemma}
Let $\cY_\omega$ be a toric stack defined as a GIT quotient as in Section~\ref{ss:GIT}.  For fixed $g, n, d$, and for $1 \leq i \leq n$, the evaluation map $\op{ev}_i: \sMbar_{g,n}(\cY_\omega, d) \to \ol I\cY_\omega$ is proper.  
\end{lemma}

\begin{proof}
The toric stack $\cY_\omega$ is defined as a stack GIT quotient $[\CC^m \sslash_\omega K]$, which maps to the coarse underlying toric variety $\CC^m \sslash_\omega K$.  The latter is projective over the affine space $(\CC^m)^K$.  Given a semi-stable orbifold curve $\cC$, a morphism $\cC \to (\CC^m)^K$ must be constant.  Any stable map $\cC \to \cY_\omega$ lies in a fiber of $\cY_\omega \to (\CC^m)^K$.
Therefore
the space $\sMbar_{g,n}(\cY_\omega, d)$ coincides with the moduli space of stable maps to $\cY_\omega$ relative to $(\CC^m)^K$ (in the sense of \cite{AV}).  This moduli space is denoted as $\mathcal K_{g,n}(\cY_\omega/ (\CC^m)^K, d)$.  It is shown \cite{AV} that $\mathcal K_{g,n}(\cY_\omega/ (\CC^m)^K, d)$ is proper over $(\CC^m)^K$.  The evaluation map $\op{ev}_i: \sMbar_{g,n}(\cY_\omega, d) \to \ol I\cY_\omega$ fits into the diagram 
\[
\begin{tikzcd}[row sep=small]
 \sMbar_{g,n}(\cY_\omega, d) \ar[r, "\op{ev}_i"] \ar[ddr] & \bar I\cY_\omega \ar[d] \\
 & \CC^m \sslash_\omega K \ar[d] \\
& (\CC^m)^K.
\end{tikzcd}
\]
The  map  $\sMbar_{g,n}(\cY_\omega, d) \to (\CC^m)^K$ 
is proper, therefore $\op{ev}_i$ is as well, by \cite[\href{https://stacks.math.columbia.edu/tag/01W6}{Lemma 01W6}]{stacks-project}.
\end{proof}

\subsection{Hypersurfaces}
Let $\cX_\omega$ be a smooth and proper toric Deligne--Mumford stack.
Choose a divisor $D = \sum_{i=0}^m a_i D_i$, with $a_i \in \ZZ$.  
Let $$\phi_D: \bn \otimes \RR  \to \RR$$ be the \emph{support function} for $D$, a piecewise-linear function which is linear on each cone of $\Sigma_\omega$ and characterized by the condition that $\phi_D(\ol b_i) = -a_i$ for $i \notin S$.  (See \cite[Chapters~4 and~6]{CLS} for a discussion of the support function.)

We make the following assumptions on $D$.
\begin{condition}\label{a:as1} 
The function $\phi_D$ satisfies the following:
 \begin{enumerate}
\item \label{i:as11} For each $\sigma \in \Sigma_\omega$,  there exists an element $m_\sigma \in \bn^\vee$ such that $$\phi_D(n) = \langle m_\sigma, n \rangle$$ for $n \in |\sigma|$;
\item \label{i:as12} The graph of $\phi_D$ is convex and $\phi_D(\ol b_i) \geq - a_i$ for $i \in S$.  
\end{enumerate}
\end{condition}
See \cite[Chapters 14 and 15]{CLS} for details on these conditions. Our purpose in making these assumptions is explained in the following lemma. 
\begin{lemma}\label{l:Dconv} 
Under the assumptions above, $D$ is basepoint free and $\cc O_{\cX_\omega}(D)$ is a convex line bundle on ${\cX_\omega}$.
\end{lemma}

\begin{proof}
By Part~\eqref{i:as11} of Condition~\ref{a:as1}, $D$ is pulled back from a Cartier divisor $\underline D$ on the coarse space $X_\omega$.
See \cite[Theorem~15.1.1 and Proposition~15.1.3]{CLS} for the proof when $\Sigma_\omega$ is a fan rather than a stacky fan.  The general case follows from a similar argument.  By \cite[Proposition~15.1.3]{CLS}, Part~\eqref{i:as12} of Condition~\ref{a:as1} implies that  $\underline D$ is nef.  By \cite[Theorem~15.1.1]{CLS}, $\underline D$ is basepoint free and thus so is $D$.

Let $f: \cC \to {\cX_\omega}$ be a stable map from a genus-zero $n$-marked orbi-curve $\cC$.
Let $r: \cC \to C$ be the map to the underlying coarse curve.  
By   \cite[Theorem 1.4.1]{AV}, the map $ \cC \to {\cX_\omega} \to X_\omega$ factors through a map $\underline{f}: C \to X_\omega$.  Therefore $f^*\cc O_{\cX_\omega}(D) = r^* \underline{f}^* \cc O_{X}(\underline D)$.  We observe that
\begin{align*} H^1\left(\cC, f^*\cc O_{\cX_\omega}(D)\right)&=H^1\left(C, r_*(f^*\cc O_{\cX_\omega}(D))\right) \\&= H^1\left(C, r_*(r^*\underline{f}^*\cc O_{X_\omega}(\underline D))\right) \\
&= H^1\left(C, \underline{f}^*\cc O_{X_\omega}(\underline D) \otimes r_*(\cc O_{\cC})\right)\\
&=
H^1\left(C, \underline{f}^*\cc O_{X_\omega}(\underline D)\right),
\end{align*}
where the third equality is the projection formula.

Since $\underline D$ is basepoint free, we have a map $|\underline D|: X_\omega \to \PP^M$ such that $\cc O_{X_\omega}(\underline D)$ is the pullback of $\cc O_{\PP^M}(1)$.  We conclude that
\begin{align*}
H^1\left(C, \underline{f}^*\cc O_{X_\omega}(\underline D)\right)
=
H^1\left(C, \underline{f}^*\circ |\underline D|^*\cc O_{\PP^M}(1) \right).
\end{align*}
As $\cc O_{\PP^M}(1)$ is convex, the above is zero.
\end{proof}

\begin{corollary}\label{c:bert}
For  a general section $s \in \Gamma(\cX_\omega, \cc O_{\cX_\omega}(D))$, The zero locus $Z(s)$ is a smooth orbifold.
\end{corollary}
\begin{proof}
By Lemma~\ref{l:Dconv} $D$ is basepoint free.  The result then follows from a general version of the Bertini theorem \cite[2.8]{Otta} and \cite[Theorem~2.1]{Wang}.
\end{proof}
As another consequence of Lemma~\ref{l:Dconv}, the results of Section~\ref{s:qqsd} hold for these hypersurfaces.

 For future reference we record the following facts about $\cc O_{\cX_\omega}(D)$.  We may express a section $s \in \Gamma(\cX_\omega, \cc O_{\cX_\omega}(D))$ as 
\[s = \sum_{j} c^j m_j\]
where each $m_j$ is a torus invariant section of $D$ and $c^j \in \CC$.  Define
\[\Delta_{D} := \{m \in (\bn \otimes \RR)^\vee| \langle m, n\rangle \geq \phi_D(n) \text{ for all } n \in \bn\otimes \RR\}.  \]
 Then the  torus invariant sections $m_j$ are identified with lattice points in $\Delta_D$.  Under this correspondence, the element $m_j \in \bn^\vee$ is identified with the rational function
 $$\prod_{i=1}^m x_i^{\langle m_j, b_i\rangle} = \prod_{i=1}^m \frac{x_i^{\langle m_j, b_i\rangle + a_i}}{x_i^{a_i}}.$$
 Note that given $m_j \in \bn^\vee$, \emph{viewed as a (rational) section of $\cc O_{\cX_\omega} (D)$}, the order of vanishing of $m_j$ along $D_i$ is then 
 \begin{equation}\label{e:oov} \op{ord}_{D_i}(m_j) = \langle m_j, b_i \rangle + a_i.\end{equation}


\section{Geometric setup for extremal transitions}\label{s:ts}

\begin{definition}[\cite{MTh}]\label{d:extr}
Given two smooth projective varieties $\cZ$ and $\wt \cZ$, we say they are related by an \emph{extremal transition} if there exists a singular variety $\ol \cZ$ together with
\begin{itemize}
\item a projective degeneration $\cc S \to \Delta$ of $\cZ$ to $\ol \cZ =  \cc S_0$;
\item a projective crepant resolution $\psi: \wt \cZ \to \ol \cZ$ (i.e. a birational morphism such that $\psi^*(K_{\ol \cZ}) = K_{\wt \cZ}$).
\end{itemize}

We may extend the above definition to the case where $\cZ$ and $\wt \cZ$ are smooth Deligne--Mumford stacks with projective coarse moduli space.  In this case we require that $\psi$ is an orbifold crepant resolution of $\cZ$, i.e. $\wt \cZ$ is smooth as a stack.
\end{definition}

In this section we give a general construction which yields extremal transitions between toric hypersurfaces.  This is related to a well-known construction \cite{CGGK, ACJM} of transitions between toric Calabi--Yau hypersurfaces.  However here we focus on toric blowups, while also generalizing beyond the Calabi--Yau setting.

\subsection{Toric blow-ups}\label{s:tbl}

Let $K$ be a torus and let $D_1, \ldots, D_m \in \Hom(K, \CC^*)$ denote a set of characters.  Choose a stability condition $\omega$ to obtain the toric variety $\cX = \cX_\omega$ as in Section~\ref{s:tor}.  Consider a 
toric subvariety  $\cV$ defined by the vanishing of 
a subset of the homogeneous coordinates $x_1, \ldots, x_m$.
  By possibly reordering the divisors $D_1, \ldots, D_m$, we can assume without loss of generality that $\cV$ is defined by the vanishing of the first $k$ coordinates.   To avoid a trivial situation we assume $\langle \rho_1, \ldots, \rho_k \rangle$ is a cone of $\Sigma_\omega$ - in particular $\{1, \ldots, k\} \cap S = \emptyset$.

We 
may realize the map 
\[\text{Bl}_\cV(\cX) \to \cX\] as an instance of toric wall crossing. 
Let $\wh K$ denote the torus $K \times \CC^*$ and let $\wh \LL := \Hom(\CC^*, \wh K) \cong \LL \oplus \ZZ$.

For $1 \leq i \leq m$ we define $\wh D_i \in \wh \LL^\vee$ as follows:
\begin{equation}\label{e:dhat}
\wh D_i := \left\{ \begin{array}{ll} (D_i, 1) & \text{if $i \leq k$} \\
(D_i ,0) & \text{if $k < i \leq m$} \\
\end{array}\right.
\end{equation}
and define
$$\wh D_\bee := (\mathbf{0}, -1).$$

Let $\beta: \ZZ^m \to \bn$ be the map from \eqref{e:des}.  We extend this to a map $\beta: \ZZ^{m+1} \to \bn$ by defining $b_\bee =\beta(e_{m+1}) := \sum_{i=1}^k \beta(e_i)$.  The following sequence is exact.
 $$0 \to \LL \xrightarrow{(D_1, \ldots, D_m, D_\bee)} \ZZ^{m+1} \xrightarrow{\beta} \bn \to 0.$$

Choose $0 < \epsilon \ll 1$ and define \begin{align*}\omega_+ &:= \omega \oplus (\epsilon) \\ \omega_- &:= \omega \oplus (- \epsilon).\end{align*} 
We will consider the GIT quotients $[\CC^{m+1} \sslash_{\omega_{+/-}} \wh K]$, where the action of $K$ on $\CC^{m+1}$ is given by $(\wh D_1, \ldots, \wh D_m, \wh D_\bee)^{T}$.  Let $\sA_{\omega_{+,-}} \subset \mathcal{P}(\{1, \ldots, m, \bee\})$ denote the corresponding sets of anticones.    Let $\Sigma_{\omega_{+/-}} \subset \bn \otimes \RR$ denote the respective fans.
\begin{proposition}
The toric stack $[\CC^{m+1} \sslash_{\omega_-} \wh K]$ is equal to $\cX$.
\end{proposition}
\begin{proof}
This follows immediately from the observation that $J$ is an anticone of $\sA_\omega$ if and only if $J \cup \{m+1\}$ is an anticone of $\sA_{\omega_-}$.  Therefore $\Sigma_{\omega_-} = \Sigma_{\omega}$.
\end{proof}

On the other hand,
define 
\[\wt \cX  := [\CC^{m+1} \sslash_{\omega_+} \wh K].\]  
\begin{proposition}
The toric stack $\wt{\cX}$ is equal to $Bl_\cV(\cX)$.  More precisely, the fan $\Sigma_{\omega_+}$ for $\wt \cX$ contains exactly the following maximal cones:
\begin{enumerate}
\item For every maximal cone $\sigma = \op{cone} \{ \ol b_i\}_{i \in I}$  of $\Sigma_{\omega}$ such that $I$ does not contain $\{1, \ldots, k\}$, the  cone $\sigma$ is in $\Sigma_{\omega_+}$.
\item For every maximal cone $\sigma = \op{cone} \{ \ol b_i\}_{i \in I}$ of $\Sigma_{\omega}$ such that $I$ contains $\{ 1, \ldots,  k\} $ and  for each $1 \leq \hat i \leq k$, the  cone $\op{cone}\{\ol b_i\}_{i \in I \setminus \{\hat i \}} \cup \{ \ol b_\bee\}$ is in $\Sigma_{\omega_+}$. 
\end{enumerate}
\end{proposition}  
\begin{proof}
Consider $\sigma = \sigma_I$ a maximal cone of $\Sigma_{\omega}$.  Then $J = I^{\op{c}}$ is a minimal anti-cone of $\sA_{\omega}$.  Assume first that $\{1, \ldots, k\}$ is not contained in $I$.  Then $\{1, \ldots, k\} \cap J \neq \emptyset$.
 By assumption, there are constants $c_j > 0$ such that 
\[\sum_{j \in J} c_j D_j = \omega.\]
Furthermore, 
\[ \sum_{j \in J} c_j \wh D_j = \left(\omega, \sum_{j \in \{1, \ldots, k\} \cap J} c_j\right).\]
By shrinking $\epsilon$ if necessary, we may assume that $$\epsilon < \sum_{j \in \{1, \ldots, k\} \cap J} c_j.$$  Then 
\[\sum_{j \in J} c_j \wh D_j +  \left(- \epsilon+ \sum_{j \in \{1, \ldots, k\} \cap J} c_j\right)\wh D_\bee = (\omega, \epsilon).\]
Therefore $J \cup \{\bee\}$ is a minimal anticone of $\sA_{\omega_+}$.

Next, assume that $\{1, \ldots, k\}$ is contained in $I$.  Then the anticone $J = I^{\op{c}}$ is disjoint from $\{1, \ldots, k\}$.  As before,
 there are constants $c_j > 0$ such that 
\[\sum_{j \in J} c_j D_j = \omega.\]
Therefore, 
\[\sum_{j \in J} c_j \wh D_j + \epsilon \wh D_{\hat i} = \left(\omega + \epsilon D_{\hat i}, \epsilon\right)\] for $1 \leq \hat i \leq k$.
By Lemma~\ref{l:wiggle} below with $v_1 = \wh D_{\hat i}$ and $v_2 = ({\mathbf{0}}, 1)$, the stability conditions $\left(\omega + \epsilon D_{\hat i}, \epsilon\right)$ and $(\omega, \epsilon)$ lie in the same maximal chamber of $\wh \LL^\vee \otimes \RR$.  Therefore $J \cup \{\hat i\}$ is a minimal anticone of $\sA_{\omega_+}$. 

We observe that the union of the support of each of the cones described above is equal to $|\Sigma_{\omega}|= \bn \otimes \RR$.  This guarantees that there are no maximal cones of $\Sigma_{\omega_+}$ other than those described above.
\end{proof}

\begin{lemma}\label{l:wiggle}
Let 
$\omega$ and  $\wh D_1, \ldots, \wh D_m, \wh D_\bee$ be as above.  Consider the wall and chamber structure on $\wh \LL^\vee \otimes \RR$ determined by $\{\wh D_1, \ldots,  \wh D_m, \wh D_\bee\}$.

Let $v_1$ and $v_2$ in $\wh \LL^\vee \otimes \RR$  be vectors with last coordinates either both positive or both negative.  Consider the stability conditions $\omega_i(\epsilon) := (\omega, 0) + \epsilon v_i$ for $i = 1, 2$.  Then for $\epsilon$ positive and sufficiently small, $\omega_1(\epsilon)$ and $\omega_2(\epsilon)$ lie in the same maximal chamber.
\end{lemma}

\begin{proof}  
Let $\ol C_\omega \subset \LL^\vee \otimes \RR$ denote the closure of the (maximal) chamber containing $\omega$.
We first claim that with respect to the wall and chamber structure on $ \wh \LL^\vee \otimes \RR$ determined by $\{\wh D_1, \ldots,  \wh D_m, \wh D_\bee\}$, the stability condition $(\omega, 0)$ is not contained in the closure of any wall other than $\ol C_\omega \oplus (0)$.

  To see this, assume the contrary, 
that $W' \subset   \wh \LL^\vee \otimes \RR$ is a wall such that $(\omega, 0) \in \ol{W' }$.  If $W'  \subset \LL^\vee \otimes \RR \oplus (0)$ then $W' = W'' \oplus (0)$ for $W''  \subset \LL^\vee \otimes \RR$ a wall with respect the wall and chamber structure determined by $\{D_1, \ldots, D_m\}$.  By Assumption~\ref{a:stab}, $W''$ must be dimension $r$ and therefore must be $C_\omega$.  So we may assume that $W'$ is not a subset of  $\LL^\vee \otimes \RR \oplus (0)$.  

By assumption $W'$ has dimension less than $r+1$.  Therefore there must exist a set $J = \{j_1, \ldots, j_s\} \subset \{1, \ldots, m, \bee\}$ with $s< r+1$ and  non-negative constants $c_1, \ldots, c_s$ such that 
$$ (\omega, 0) = \sum_{i=1}^s c_i \wh D_{j_i}.$$
Because $W'$ is a wall not contained in $\LL^\vee \otimes \RR \oplus (0)$, at least one of $
\wh D_{j_i}$ has a nonzero last coordinate which then forces $\bee$ to be in $J$.  Without loss of generality we may assume $j_s = \bee$.  We conclude that 
$$\omega = \sum_{i=1}^{s-1} c_i D_{j_i}.$$
Because $s-1 < r$, this contradicts Assumption~\ref{a:stab}.  This proves the claim

We now show that for $\epsilon$ positive and sufficiently small, $\omega_i(\epsilon)$ lies in a maximal chamber.  If we assume the contrary, then the line segment $$S_i := \{\omega_i(\epsilon)| 0< \epsilon \ll 1\}$$ must be contained in a wall $W_i$.  By assumption on $v_i$, $S_i$ is not contained in $\LL^\vee \otimes \RR \oplus (0)$.  But this contradicts the fact that $(\omega, 0)$ is not contained in the closure of any wall other than $C_\omega$.  

To show that $\omega_1(\epsilon)$ and $\omega_2(\epsilon)$ lie in the same maximal chamber, we must show there is not a wall through $(\omega, 0)$ which separates $\omega_1(\epsilon)$ and $\omega_2(\epsilon)$ for $\epsilon$ positive and sufficiently small.  By assumption, the sign of the last coordinate of $\omega_1(\epsilon)$ and $\omega_2(\epsilon)$ is the same, so the wall $\LL^\vee \otimes \RR \oplus (0)$ does not separate them.  Then again by the fact that $(\omega, 0)$ is not contained in the closure of any wall other than $C_\omega$, we reach the desired conclusion.
\end{proof}

By the above two propositions we see that a toric blow-up is an example of a (discrepant) toric wall crossing. 
In this case the wall $W$ is generated by the vector $(\mathbf{0}, 1)$.  

\subsection{Extremal transition}\label{ss:et}
Choose a divisor $D = \sum_{i=0}^m a_i D_i$ on $\cX$ satisfying Condition~\ref{a:as1}.  Let $s \in \Gamma(\cX, \cc O_\cX(D))$ be a general section. Define the hypersurface $$\cZ:= Z(s).$$  This is a smooth Deligne--Mumford stack by Corollary~\ref{c:bert}.  Recall that we may express the section $s$ as 
$$s = \sum_{j} c^j m_j,$$
where the sum ranges over all of the  lattice points $m_j$ in $\Delta_D$
%
%

We may degenerate the section to some $\ol{s}$ by setting some of the $c^j$ to zero.  
Define
\[ \ol{c^j} := \left\{ \begin{array}{ll} c^j & \text{if $\sum_{i=1}^k \op{ord}_{D_i} m_j \geq k-1$
} \\
0 & \text{otherwise,}
\end{array}\right.\]
where $\op{ord}_{D_i} m_j$ is given by \eqref{e:oov}.
Consider the section $\ol s = \sum \ol{c^j} m_j$. 
Then $$\ol{\cZ} := Z(\ol{s})$$ is a degeneration of $\cZ$.  It will contain (and will usually develop singularities on) $\cV$.
One can often find a resolution of $\ol \cZ$ as follows.

Let $q: \wt  \cX \to \cX$ denote the blowup map.
The pullback of $D$ to $\wt \cX$ is given by 
\begin{align}\label{e:pb}
q^*(D) &= \sum_{i = 1}^m a_i \wh D_i - \phi_D(\ol b_\bee)\wh D_{\bee} \\ \nonumber
&= \sum_{i = 1}^m a_i \wh D_i - \phi_D \left(\sum_{i=1}^k \ol b_i \right) \wh D_{\bee}
\\ \nonumber &= \sum_{i = 1}^m a_i \wh D_i + \left(\sum_{i=1}^k a_i \right) \wh D_{\bee}.\end{align}  
Define the divisor $\wt D$ to be 
\begin{equation}\label{e:wtd}\wt D := \sum_{i = 1}^m a_i \wh D_i + \left(1-k + \sum_{i=1}^k a_i \right) \wh D_{\bee}.\end{equation}

Note that $\Delta_{\wt D}$ is a subset of  $\Delta_D$.  Precisely, the element $m_j \in \Delta_D$ lies in $\Delta_{\wt D}$ if and only if 
$$\langle m_j, b_\bee \rangle + \sum_{i=1}^k a_i \geq k-1.$$ The left hand side of the above equation is equal to $\sum_{i=1}^k \op{ord}_{D_i} m_j$ (viewing $m_j$ as a section of $D$), and therefore
 $\sum \ol{c^j} m_j$ defines a general section of $\wt D$.  To avoid confusion with $\ol{s}$, we denote this section by $\tilde s \in \Gamma(\wt \cX, \wt D)$, although both $\ol{s}$ and $\tilde s$ may be expressed as $\sum \ol{c^j} m_j$.

We will always assume Condition~\ref{a:as1} holds for $\wt D$ as well as for $D$.
Under this assumption, if we have chosen $\{\ol{c^j} \}$ sufficiently general, then 
$$\wt{\cZ} := Z(\tilde s)$$
will be a smooth variety (or orbifold). $\wt{\cZ}$ may also be described as
the proper transform of $\ol{\cZ}$ under the map $q: \wt \cX \to \cX$.

\begin{proposition}
The hypersurfaces $\cZ$ and $\wt{\cZ}$ are  related by an extremal transition through $\ol \cZ$ in the sense of Definition~\ref{d:extr}.
\end{proposition}
\begin{proof}
We must check that the resolution $\psi: \wt \cZ \to \ol \cZ$ is crepant.

We first claim that the singular locus of $\ol \cZ$ is codimension at least 2.  Note that $\ol \cZ$ and $\wt \cZ$ are isomorphic outside of $\cV$, so the claim is immediate if the codimension of $\cV$ in $\cX$ is at least three.  We are left with the case that $\cV$ is codimension 2 in $\cX$ (and therefore codimension 1 in $\ol \cZ$).  In this case, the section $\ol s$ is defined to be a general section which vanishes on $\cV$.  We may therefore express $\ol s$ as 
$$\ol s = x_1  f(\underline x) + x_2 g(\underline x).$$
where $f$ and $g$ are general sections of $\cc O_\cX(D - D_1)$ and $\cc O_\cX(D - D_2)$ respectively.  A local coordinate computation shows that the singular locus is given by 
$$\{x_1 = x_2 = f(\underline x) = g(\underline x) = 0\}.$$ 
The claim will follow if we can show that at least one of $f$ or $g$ is non-vanishing on $\cV$.  By Condition~\ref{a:as1}, $\wt D$ is basepoint free.  There must therefore exist at least one torus-invariant section $m_j \in \Gamma(\wt \cX, \cc O_{\wt \cX}(\wt D))$ which is non-vanishing on $D_\bee$.  Representing $m_j$ as a lattice point in $\Delta_{\wt D} \subset \bn^\vee \otimes \RR$, by \eqref{e:wtd} and \eqref{e:oov} we have that 
$$\langle m_j, b_\bee \rangle + a_1 + a_2 - 1 = 0.$$
On the other hand, since $m_j$ also lies in  $\Delta_D$, 
$$\langle m_j, b_i \rangle + a_i \geq 0$$
for $i = 1, 2$.  Since $b_\bee = b_1 + b_2$, it follows that for $i$ equal one of $1$ or $2$, 
$\langle m_j, b_i \rangle + a_i = 1,$ and for the other $\langle m_j, b_i \rangle + a_i = 0.$
Without loss of generality, we assume that $\langle m_j, b_1 \rangle + a_1 = 1,$.  Then $m_j$ corresponds to a torus-invariant section of $D$ which vanishes to order 1 at $D_1$ and order 0 at $D_2$.  In homogeneous coordinates, it may therefore be written as $x_1f_j(\underline x)$ where $f_j(\underline x)$ is nonvanishing on $\cV$.
We conclude that a general section $f$ of $\cc O_\cX(D - D_1)$ is nonvanishing on $\cV$.  
This proves the claim that the singular locus of $\ol \cZ$ is codimension at least 2.  

We may therefore calculate the canonical class on the complement of the singular locus.  Let $\ol \cS \subset \cV$ denote the singular locus of $\ol \cZ$.  Let $\cU$ denote the complement $\cX \setminus \ol \cS$.  Then $\ol s|_U$ is a regular section of $\cc O_U(D)$, whose zero locus is smooth.  Applying the adjunction formula to $\ol \cZ \setminus \ol \cS \hookrightarrow \cU$ and then extending to all of $\ol \cZ$, we conclude that 
$$ K_{\ol \cZ} = \ol k^*(K_\cX + D),$$
exactly as if $\ol \cZ$ were smooth.  

Consider the commutative diagram
\[
\begin{tikzcd}
\wt \cZ \ar[r, "\wt k"] \ar[d, "\psi"] & \wt \cX \ar[d, "q"] \\
\ol \cZ \ar[r, "\ol k"] & \cX.
\end{tikzcd}
\]
We must show that $p^*(K_{\ol \cZ}) = K_{\wt \cZ}$.  Again applying the adjunction formula, 
$$K_{\wt \cZ} = \wt k^*(K_{\wt \cX} + \wt D).$$  It therefore suffices to show that 
 \begin{equation}\label{e:divrel} q^*(K_\cX + D) = K_{\wt \cX} + \wt D.\end{equation}
The pullback $q^*(K_\cX) = q^*(- \sum_{i=1}^m D_i)$ is equal to 
\begin{align*} -\sum_{i=1}^m  \wh D_i - \phi_{K_\cX}(\ol b_\bee)\wh D_\bee &= - \sum_{i=1}^m \wh D_i - \phi_{K_\cX}\left(\sum_{i=1}^k \ol b_i\right)\wh D_\bee \\
&= - \sum_{i=1}^m \wh D_i  - k \wh D_\bee \\
& = K_{\wt \cX} + (1-k) \wh D_\bee.
\end{align*}
On the other hand, by \eqref{e:pb} and \eqref{e:wtd}, 
$$ q^*(D) = \wt D + (k-1) \wh D_\bee.$$
Equation~\eqref{e:divrel} follows.
\end{proof}
\begin{remark}
The construction of extremal transitions given above can be generalized to the case where the map $\wt \cX \to \cX$ between the ambient spaces is a \emph{weighted blow-up}.  In this modification, we replace $\wh D_i$ from \eqref{e:dhat} with $(D_i, w_i)$ for $1 \leq i \leq k$, where $w_i$ are positive integers, and we replace \eqref{e:wtd} with 
$$ \wt D := \sum_{i = 1}^m a_i \wh D_i + \left(1 + \sum_{i=1}^k w_i(a_i-1)\right) \wh D_{\bee}.$$ 
The results of the paper hold in this context as well, with the same proofs, although the notation is slightly more cumbersome.  We leave the details to the reader.
\end{remark}

\subsection{Examples} 
This setup includes many well-known examples.
\begin{example}[Conifold transition]
Let $K = \CC^*$, $m= 5$,  $k = 2$, and $D_1 = \cdots = D_5 = \omega = 1 \in \Hom(K, \CC^*) \cong \ZZ$.  Define
$$D = \sum_{i=1}^5 D_i= -K_\cX.$$
Then $\cZ = Z(s)$ is a smooth quintic 3-fold in $\cX = \PP^4$.  
Degenerating the section $s$, we obtain a section  $\bar s$ which may be written as $x_1 f(\underline x) + x_2 g(\underline x)$.  The variety $\ol \cZ = Z(\bar s)$ contains $\PP^2 = Z(x_1,  x_2 )$ and has 16 nodes at 
$$x_1 = x_2 = f(\underline x) = g(\underline x) = 0.$$ 

The variety $\wt \cX$ is equal to $\op{Bl}_{\PP^2} \PP^4$.  The divisor $\wt D$ is equal to $\sum_{i=1}^5 \wh D_i + \wh D_6 = -K_{\wt \cX}$, and the section $\wt s$ defines a smooth Calabi--Yau $\wt \cZ$ which resolves the singularities of $\ol \cZ$.  The map $\wt \cZ \to \cZ$ contracts curves.  The transition between $\cZ$ and $\wt \cZ$ is a conifold transition.

One checks explicitly that $D$ and $\wt D$ both satisfy Condition~\ref{a:as1}.
\end{example}

\begin{example}\label{e:proj}

More generally we can let $K = \CC^*$, choose any $m$ and $k$ with  $1<k<m$, and let $D_1 = \cdots = D_m = \omega = 1$ .  For any choice of $a_1, \ldots, a_m$, the divisor $$D = \sum_{i=1}^m a_i D_i \cong  \left(\sum_{i=1}^m  a_i\right) H$$  satisfies Condition~\ref{a:as1} if $d = \sum_{i=1}^m  a_i$ is positive.
Here $H$ is a hyperplane.
The section $s$ defines a degree $c$ hypersurface 
$\cZ$ of  dimension $m-2$ in $\cX = \PP^{m-1}$, and $\ol \cZ = Z(\bar s)$ is a singular hypersurface which vanishes at $\{x_1 = \cdots = x_k\} \cong \PP^{m-k-1}$.  In this case $\wt \cX = \op{Bl}_{\PP^{m-k-1}} \PP^{m-1}$, and $$\wt D \cong d H + (1-k) E, $$ where $H$ is the pullback of a generic hyperplane in $\PP^{m-1}$ and $E$ is the exceptional divisor.    Since $\wt \cX$ is a smooth variety, Part~\eqref{i:as11} of Condition~\ref{a:as1} is automatic for $\wt D$.  By \cite[Theorem~15.1.1]{CLS},
the divisor $\wt D$ will  satisfy Part~\eqref{i:as12} of Condition~\ref{a:as1} provided $\wt D$ is nef. The nef cone for $\wt \cX$ is generated by the rays $\wh D_1 = (1, 1)$ and $\wh D_m = (1,0)$.  Condition~\ref{a:as1} is therefore satisfied if $(d, k-1)$ lies in the cone generated by $(1,1)$ and $(1,0)$, or equivalently, if
$$d \geq k-1.$$ 

In this way we generate numerous examples of transitions in any dimension for hypersurfaces of any degree.  The case of $m=5, k=4,$ and $ a_1 = \cdots = a_5 = 1$ gives the cubic transition studied by the first author in \cite{MiCub}.
\end{example}

\begin{example} There is no particular reason to restrict to $\cX$ a projective space.  
Generalizing the above example, we consider a product of projective spaces
$$\cX = \PP^{m_1 - 1} \times \cdots \times \PP^{m_r - 1},$$
by choosing each $D_i$ from $\{ e_j\}_{1\leq j\leq r}$, where $e_j = (0, \ldots, 0, 1, 0, \ldots, 0)$ is the vector with $1$ in the $j$th coordinate and zeroes in all other coordinates.  

By reordering $\{D_i\}_{1 \leq i \leq m}$, the set $\cV$ may be given by 
$$\bigcap_{j=1}^r \{x_{j,1} = \cdots = x_{j, k_j} = 0\},$$
where $\{x_{j,i}\}_{1 \leq i \leq m_j}$ are homogeneous coordinates on the $j$th factor of $\cX$, $\PP^{m_j - 1}$.  
After choosing appropriate $a_1, \ldots, a_m$, the divisor 
$D$ is conjugate to $ \sum_{j=1}^r d_j H_j,$
where $H_j$ is the pullback of the hyperplane class from $\PP^{m_j - 1}$.  Then $D$ will satisfy Condition~\ref{a:as1} if $d_j \geq 0$ for all $j$, and $\wt D$ will satisfy Condition~\ref{a:as1} if for all $j$ with $k_j \neq 0$,
$$d_j \geq \left(\sum_{l=1}^r k_l\right)-1.$$

\end{example}

\begin{example}\label{e:wp}
For an orbifold example, let $K = \CC^*$.  Choose positive integers $c_1, \ldots, c_{m-1}$, let $D_i = c_i$ for $1 \leq i < m$, and let $D_m = \omega = 1$.  Then $\cX$ is the weighted projective space $\PP(c_1, \ldots, c_{m-1}, 1)$.  
The fan for $\cX$ lies in $\RR^{m-1}$ and contains the rays $b_i = e_i = (0, \ldots, 1, \ldots, 0)$ for $1 \leq i <m$ and $b_m = (-c_1, \ldots, -c_{m-1})$.
Let $a_1 = \cdots = a_{m-1} = 0$ and $a_m = d > 0$.  If $d$ is divisible by each of $c_1, \ldots, c_{m-1}$ then $D$ will satisfy Condition~\ref{a:as1}.  The divisor $\wt D$ will satisfy Condition~\ref{a:as1} if 
$$d \geq \hat c (k-1)$$
where $\hat c := \op{max} \left( \{c_1, \ldots, c_k\}\right)$.
\end{example}

\section{The total spaces}\label{s:tots}
We want to compare the Gromov--Witten theory of $\cZ$ and $\wt{\cZ}$, as defined in Section~\ref{s:ts}.  By quantum Serre duality (Theorem~\ref{t:qsd}), it suffices to compare the Gromov--Witten theory of the total spaces of the line bundles $-D$ and $- \wt D$.  Define 
\[ \cT := \tot( \cc O_{\cX} (- D)) \text{ and }  \wt \cT := \tot( \cc O_{\wt \cX}(-\wt D)).\]
We can represent both $\cT$ and $\wt \cT$ as a toric GIT quotient.  We will focus on $\wt \cT$.

Let $\wh K$ denote the torus $K \times \CC^*$ and recall the definition of $\wh D_1, \ldots , \wh D_m, \wh D_\bee$ from Section~\ref{s:tbl}. 
Define 
\[\wh D_{\bff} := -\sum_{i=1}^{m} a_i \wh D_i - \left(1-k + \sum_{i=1}^k a_i \right) \wh D_{\bee} = \left(-\sum_{i=1}^m a_i D_i, 1-k\right).\]
We will consider the GIT quotients $[\CC^{m+2} \sslash_{\omega_{+/-}} \wh K]$, where the action of $\wh K$ on $\CC^{m+2}$ is given by $(\wh D_1, \ldots, \wh D_m, \wh D_\bee, \wh D_\bff)$.  Let $\wh \sA_{\omega_{+,-}} \subset \mathcal{P}(\{1, \ldots, m, \bee, \bff\})$ denote the corresponding sets of anticones.

Define $\wh \beta: \ZZ^{m+2} \to \wh \bn := \bn \oplus \ZZ$ by 
\[
\wh \beta (e_i) := \left\{ \begin{array}{ll} (\beta(e_i) , a_i) & \text{if $1 \leq i \leq m$} \\
\left( \sum_{j=1}^k \beta(e_j),  1-k +  \sum_{j=1}^k a_j \right) & \text{if $i= \bee$} \\
(\mathbf{0}, 1) & \text{if $i = \bff$},
\end{array}\right.
\] where $\beta: \ZZ^m \to \bn$ is the map from \eqref{e:des} for $\cX = \cX_\omega$.
A simple check shows that 
\begin{equation}\label{e:des2} 0 \to \wh \LL \xrightarrow{(\wh D_1, \ldots, \wh D_m, \wh D_\bee, \wh D_\bff)} \ZZ^{m+2} \xrightarrow{\wh \beta} \wh \bn  \to 0,
\end{equation}
is exact.  Define $\wh b_i := \wh \beta(e_i)$.  We denote the fans corresponding to the GIT quotients $[\CC^{m+2} \sslash_{\omega_{+/-}} \wh K]$
by $\wh \Sigma_{\omega_{+/-}}$.
\begin{proposition}\label{p:wtt}
The total space $\wt \cT$ may be expressed as the toric GIT quotient $[\CC^{m+2} \sslash_{\omega_+} \wh K]$.  Furthermore $j \in S_{\omega}$ if and only if $j \in S_{\omega_+}$. 
\end{proposition}  
\begin{proof}
Given $\sigma_I = \op{cone} \{ \ol b_i\}_{i \in I}$ a cone of $\Sigma_{\omega_+}$, let $\wh \sigma_I := \op{cone} \{\ol {\wh b}_i\}_{i \in I} \cup \{{\wh b}_{\bff}\}$.
It follows immediately from considering anticones that each cone $\wh \sigma_I$ is a cone of $\wh \Sigma_{\omega_+}$.
Note that for $i \in \{1, \ldots, m, \bee\}$, 
$$\ol{\wh b_i} = \left(\ol{ b_i}, -\phi_{\wt D}(\ol b_i)\right).$$
As a consequence,
 the union 
\[ \bigcup_{\sigma \in \Sigma_{\omega_+}} |\wh \sigma|\]
 is the set of all points $\bf n \in \wh \bn \otimes \RR$ such that the last coordinate of $\bf n$ is greater than or equal to $-\phi_{\wt D}(\bf n)$.
 By the convexity assumption on the support function (Part~\eqref{i:as11} of Condition~\ref{a:as1}), this set 
is equal to 
$$\op{cone} \{\ol{\wh b_i}\}_{i \in \{1, \ldots, m, \bee, \bff\}\setminus S }
 = \op{cone} \{\ol{\wh b_i}\}_{i \in \{1, \ldots, m, \bee, \bff\}}
$$
and therefore must be equal to $|\wh \Sigma_{\omega_+}|$.  Therefore $\{\wh \sigma\}_{\sigma \in \Sigma_{\omega_+}}$ must contain every cone of $\wh \Sigma_{\omega_+}$.  

We have shown that $\sigma$ is a cone of $\Sigma_{\omega_+}$ if and only if $\wh \sigma$ is a cone of $\wh \Sigma_{\omega_+}$.  
This implies that 
$(\CC^{m+2})^{\op{ss}}(\omega_+) = (\CC^{m+1})^{\op{ss}}(\omega_+) \times \CC$ and therefore that
the map $[\CC^{m+2} \sslash_{\omega_+} \wh K] \to [\CC^{m+1} \sslash_{\omega_+} \wh K] = \wt \cX$ is a vector bundle.  To see that it is the total space of $-\wt D$, it suffices to recall that 
\[\wh D_{\bff} = - \wt D.\]
The second part of the proposition follows as well from the description of the cones of $\wh \Sigma_{\omega_+}$. 
\end{proof}
\begin{corollary}\label{c:intray}
The only interior ray of $\wh \Sigma_{\omega_+}$ is $\wh b_{\bff}$.
\end{corollary}
\begin{proof}
This follows from the description of $|\wh \Sigma_{\omega_+}|$ in the proof of Proposition~\ref{p:wtt}. \end{proof}

Next we investigate the GIT quotient with respect to the stability condition $\omega_-$.  As the next proposition shows, it is \emph{not} equal to $\cT$.
Let  
\[\ol \cT := [\CC^{m+2} \sslash_{\omega_-} \wh K].\] 

\begin{proposition}\label{p:ott}
The toric stack $\ol \cT$ is a partial compactification of $\cT$.  The fan $\wh \Sigma_{\omega_-}$ for $\ol \cT$ contains precisely the following two sets of maximal cones:
\begin{enumerate}
\item \label{i:bart1} For every maximal cone $\sigma_I$ of $\Sigma_{\omega_-}$, the  cone $\wh \sigma_I:= \op{cone} \{\ol{\wh b}_i\}_{i \in I} \cup \{\ol{\wh b}_{\bff}\}$ is in $\wh \Sigma_{\omega_-}$; 
\item \label{i:bart2} For every maximal cone $\sigma_I$ of $\Sigma_{\omega_-}$ such that $I$ contains $\{1, \ldots, k\} $, the  cone $ \sigma_I' := \op{cone}(\ol{ \wh b}_i, \ol{\wh b}_{\bee})$ is in $\wh \Sigma_{\omega_-}$.
\end{enumerate}
Furthermore $j \in S_{\omega}$ if and only if $j \in S_{\omega_-}$.
\end{proposition}  

\begin{remark}
The fan formed by the type~\eqref{i:bart1} cones gives the toric stack $\cT$.  The type~\eqref{i:bart2} cones serve to partially compactify $\cT$.  See Corollary~\ref{c:parc} below.
\end{remark}
\begin{proof}
By an identical argument as in  Proposition~\ref{p:wtt}, we see that the type~\eqref{i:bart1} cones lie in $\wh \Sigma_{\omega_-}$.

We now show that the type~\eqref{i:bart2} cones lie in $\wh \Sigma_{\omega_-}$.  
Convexity of $D$  implies that $D$ lies in the closure of the extended ample cone \cite[Theorem~15.1.1 and Proposition~15.1.3]{CLS}.
Therefore, for all $J \in \sA_\omega$, there exist a set of non-negative constants $\{d_j\}_{j \in J}$ such that 
  \[\sum_{j \in J} d_j D_j = \sum_{i=1}^m a_i D_i.\]
 Since $\wh D_j = (D_j, 0)$ for $j \notin \{1, \ldots, k\}$, we have
 \begin{equation}\label{e:span} (\mathbf{0}, -1) = \frac{1}{k-1}\left(\sum_{j \in J} d_j \wh D_j + \wh D_{\bff}\right)\end{equation}
  whenever $J$ is disjoint from $\{1, \ldots, k\}$.
Let $\sigma_I$ be a cone of $\Sigma_{\omega}$ such that $I$ contains $\{1, \ldots, k\} $. Let $J$ denote the complement $I^{\op{c}}$ in $\{1, \ldots, m\}$.  We claim that $J \cup \{\bff\}$ is an anticone of $\wh \sA_{\omega_-}$.
Choose constants $c_j > 0$ such that 
\[\sum_{j \in J} c_j D_j = \omega.\]
Using the above two equations we observe that
\begin{align*}
\sum_{j \in J} c_j \wh D_j + \frac{\epsilon}{k-1}\left(\sum_{j \in J} d_j \wh D_j + \wh D_{\bff}\right) & = (\omega, \epsilon).
\end{align*}
Therefore $J \cup \{\bff\}$ is an anticone of $\wh \sA_{\omega_-}$.  

The union of the support of all type~\eqref{i:bart1} and~\eqref{i:bart2} cones is easily seen to be equal to the support of $\wh \Sigma_{\omega_+} = \op{cone} \{\ol{\wh b_i}\}_{i \in \{1, \ldots, m, \bee, \bff\}}$, which implies there are no other maximal cones.  

The final statement of the proposition follows immediately from the description of maximal cones.
\end{proof}

Figure~\ref{f3} in the introduction provides a picture of the fans in the case of $\cX = \PP^2$, $D = -K_\cX$ and $\cV$ a point.  In this picture the bottom vertex is the origin and each line out of the origin is a primitive ray vector.

 In the table below we list the relevant toric varieties constructed in this section, together with the toric data defining them.  Recall that $\wh K = K \times \CC^*$, $\wh \LL = \LL \oplus \ZZ$, $\wh \bn = \bn \oplus \ZZ$.
  
\[
\def\arraystretch{2}
  \begin{array}{|c|c|c|}
  \hline
\text{ toric stack}&\text{divisor sequence}&\text{GIT quotient} \\
\hline
\cX & 0 \to  \LL \xrightarrow{( D_1, \ldots,  D_{m})} \ZZ^{m}\xrightarrow{\beta} \bn  \to 0 & [\CC^m \sslash_\omega K]
\\
\hline
\cX  & 0 \to \wh \LL \xrightarrow{(\wh D_1, \ldots, \wh D_{m}, \wh D_\bee)} \ZZ^{m+1}\xrightarrow{\beta} \bn  \to 0 &[\CC^{m+1} \sslash_{\omega_-} \wh K]\\
\hline
\wt \cX = \op{Bl}_\cV \cX & 0 \to \wh \LL \xrightarrow{(\wh D_1, \ldots, \wh D_{m}, \wh D_\bee)} \ZZ^{m+1}\xrightarrow{\beta} \bn  \to 0 & [\CC^{m+1} \sslash_{\omega_+} \wh K]\\
\hline
\wt \cT = \op{tot}(-\wt D) & 0 \to \wh \LL \xrightarrow{(\wh D_1, \ldots, \wh D_{m}, \wh D_\bee, \wh D_{\bff})} \ZZ^{m+2}\xrightarrow{\wh \beta} \bn  \to 0
 & [\CC^{m+2} \sslash_{\omega_+} \wh K]\\
\hline
\ol \cT &
  0 \to \wh \LL \xrightarrow{(\wh D_1, \ldots, \wh D_{m}, \wh D_\bee, \wh D_{\bff})} \ZZ^{m+2}\xrightarrow{\wh \beta} \bn  \to 0 & [\CC^{m+2} \sslash_{\omega_-} \wh K]
  \\
\hline
  \end{array}
  \]

  Recall that the connected components of $I\ol \cT$ are indexed by $ \wh \KK_{\omega_-}/\wh \LL$.  These components can be separated into two groups based on whether or not they intersect $\cT$.
   \begin{definition}
  Define the following complementary subsets of $ \wh \KK_{\omega_-}$:
  \begin{align*} \wh \KK_{\omega_-}^{\op{int}} := &\{\nu \in \wh \KK_{\omega_-} | D_{\bee} \cdot \nu \in \ZZ\}\\
  \wh \KK_{\omega_-}^{\op{frac}} := &\{\nu \in \wh \KK_{\omega_-} | D_{\bee} \cdot \nu \notin \ZZ\}.
  \end{align*} 
  \end{definition}
  
  The following observations are more or less immediate from the above description of $\ol \cT$.  We record them here for future use.
  \begin{corollary}\label{c:parc}
  The open locus $\{x_{\bee} \neq 0\} \subset \ol \cT$ is equal to $\cT$.  The locus $\{x_{\bee} \neq 0\} \subset I\ol \cT$ is equal to $I\cT$.
  
  The connected components of $I\ol \cT$ are of the following two types:
  \begin{enumerate}
  \item The elements  $\nu \in \wh \KK_{\omega_-}^{\op{int}}/\wh \LL$ index those components $\ol \cT_\nu$ of $I\ol \cT$ which are partial compactifications of a corresponding 
  component $\cT_\nu$ of $I\cT$.
  In particular the embedding $\KK_\omega \hookrightarrow \wh \KK_{\omega_-} \subset \wh \LL \otimes \QQ$ given by $f \mapsto (f,0)$ induces an isomorphism 
$\KK_\omega/\LL \cong   \wh \KK_{\omega_-}^{\op{int}}/\wh \LL$.
  \item
 For $\nu \in \wh \KK_{\omega_-}^{\op{frac}}/ \wh \LL $, the component $\ol \cT_\nu$ of $I\ol \cT$ is supported on the locus 
  $\{x_{\bee} = 0\}$.
  \end{enumerate}
  \end{corollary}
  
  \begin{proof}
  
  The first statements follow from the fact that 
  \[ [\left(\CC^{m+2} \cap \{x_{\bee} \neq 0\}\right) \sslash_{\omega_-} \wh K]   = [\CC^{m+1}  \sslash_{\omega} K],\]
  where the action of $K$ on the last factor of $\CC^{m+1}$ is given by the restriction of $\wt D_{\bff} \in \op{hom}( \wh K, \CC^*)$
  to $K  \cong K \times \{1\} \subset \wh K$.  The GIT quotient on the right hand side is easily seen to be $\cT$, by an argument similar to Proposition~\ref{p:wtt}.  
  
   It follows from the definitions that 
  \[\KK_\omega /\LL \cong \wh \KK_{\omega_-}^{\op{int}}/\wh \LL,\]
  where the isomorphism is induced by the inclusion $f \mapsto (f, 0)$. The set $\wh \KK_{\omega_-}^{\op{int}}/\wh \LL$ indexes exactly those twisted sectors which have a nontrivial intersection with $\{x_\bee \neq 0\}$.
  
  
  The elements of $\wh \KK_{\omega_-}^{\op{frac}} /\wh \LL$ correspond to components of $I \ol \cT$ supported on  $\{x_{\bee} = 0\}$.
    \end{proof}

\subsection{Comparison of narrow cohomology}
In this section we identify the narrow cohomology of $\cT$ with a subquotient of the narrow cohomology of $\ol \cT$.  We also compare the  ample and Mori cones.

\begin{lemma}\label{l:narsplit} 
The narrow cohomology  
$H^*_{\op{CR, nar}}(\ol \cT)$ contains the subspace $ u_{\bff} H^*_{\op{CR}}(\ol \cT)$. 
\end{lemma}

\begin{proof}
We work with each component of $I\ol \cT$ individually.  For $\nu \in \wh \KK_{\omega_-}^{\op{frac}} /\wh \LL$, we claim that
$u_{\bff} H^*(\ol \cT_\nu) = 0$.  The toric stack $\ol \cT_\nu$ may be represented as the GIT quotient
\[ [(\CC^{m+2})^{g_\nu} \sslash_{\omega_-} \wh K],\]
where $g_\nu \in \wh K$ is as in \eqref{e:gf}.
By assumption on $\nu$, $D_{\bee} \cdot \nu \notin \ZZ$, so the $\bee$-th (= $m+1$st) coordinate of a point in $(\CC^{m+2})^{g_\nu}$ must be zero.  It follows that a semistable point of $(\CC^{m+2})^{g_\nu}$ must be nonzero in the $m+2$nd coordinate or equivalently, that the natural section of $L_{\bff}|_{\ol \cT_\nu}$ is nowhere vanishing.  This implies that under the map $\iota_\nu: \ol \cT_\nu \to \ol \cT$, the class $u_\bff$ pulls back to zero.
The claim follows by \cite[Remark 2.14 ]{ShNar}.

Next consider the untwisted sector $\ol \cT \subset I \ol \cT$.
Because $\Sigma_{\omega_-}$ and $\Sigma_{\omega_+}$ have the same support, it follows from Corollary~\ref{c:intray} that $\ol{\wh b}_{\bff}$ is an interior ray of $\wh \Sigma_{\omega_-}$.  
By Lemma~\ref{l:narcoh} this implies 
 $ u_{\bff} H^*(\ol \cT) \subseteq H^*_{\op{nar}}(\ol \cT) $.

For $\nu \in \KK_\omega/\LL\cong \wh \KK_{\omega_-}^{\op{int}}/\wh \LL$, 
the map $\iota_\nu: \ol \cT_\nu \to \ol \cT$ is a closed embedding.  By \cite[Proposition~2.5]{ShNar}, $\iota_\nu^*$  maps $H^*_{\op{nar}}(\ol \cT)$ to $H^*_{\op{nar}}(\ol \cT_\nu)$.  From this we have $$u_{\bff} H^*(\ol \cT_\nu) = u_{\bff} \iota_\nu^*\left(H^*(\ol \cT)\right)
=  \iota_\nu^*\left(u_{\bff} H^*(\ol \cT)\right) \subseteq H^*_{\op{nar}}(\ol \cT_\nu) .$$

\end{proof}

\begin{proposition}\label{c:narco}
The pullback $i^*: H^*_{\op{CR}}( \ol \cT) \to H^*_{\op{CR}}(\cT)$ restricts to a surjective map $u_{\bff} H^*_{\op{CR}}(\ol \cT) \to  u_{\bff} H^*_{\op{CR}}(\cT)= H^*_{\op{CR, nar}}(\cT).$
\end{proposition}

\begin{proof}
For $\nu \in \wh \KK_{\omega_-}^{\op{frac}} /\wh \LL$,
the summand $u_{\bff} H^*(\ol \cT_\nu)$ is zero by the previous proof.
We must check that for $\nu \in \wh \KK_{\omega_-}^{\op{int}} /\wh \LL\cong \KK/\LL$, the pullback $i^*_\nu: H^*( \ol \cT_\nu) \to H^*(\cT_\nu)$ maps $u_{\bff} H^*_{\op{CR}}(\ol \cT_\nu) $ surjectively onto $u_{\bff} H^*_{\op{CR}}(\ol \cT)$.

For a given twisted sector $\nu \in \wh \KK_{\omega_-}^{\op{int}} /\wh \LL$, consider the Stanley-Reisner presentations of $H^*( \ol \cT_\nu) $ and $ H^*(\cT_\nu)$ as described in \eqref{e:SR}.
The pullback $i^*_\nu: H^*( \ol \cT_\nu) \to H^*(\cT_\nu)$ is given simply by setting $u_{\bee}$ to zero in the Stanley-Reisner presentation of $H^*( \ol \cT_\nu) $.  It follows that the restriction $$i^*_\nu: u_{\bff} H^*_{\op{CR}}(\ol \cT_\nu) \to  u_{\bff} H^*_{\op{CR}}(\cT_\nu)$$ is surjective.
\end{proof}

\begin{lemma}\label{l:amp}
We have the following comparisons between (co-)homology of $\cT$ and $\ol \cT$.
\begin{enumerate}
\item
The map $\LL \otimes \RR \to \wh \LL  \otimes \RR= \LL \otimes \RR \oplus \RR$ given by $l \mapsto (l,0)$ defines a canonical isomorphism 
$H_2(\ol \cT; \RR) \cong H_2(\cT; \RR) \oplus \RR$. 

\item \label{i:sec} The map $\LL^\vee \otimes \RR \to \wh \LL^\vee \otimes \RR = \LL^\vee \otimes \RR \oplus \RR$ defines a canonical isomorphism
$H^2(\ol \cT; \RR) \cong H^2(\cT; \RR) \oplus \RR$.

\item Under the isomorphism above, the extended ample cones are related as:
$$C_{\omega_-} = C_{\omega} \times \RR_{\geq 0}(\mathbf{0}, -1).$$ 
\item The ample cones and Mori cones satisfy the same relationship:
$$C_{\omega_-}' = C_{\omega}' \times \RR_{\geq 0}(\mathbf{0}, -1),$$
and
$$NE(\ol \cT) = NE(\cT) \times \RR_{\geq 0}(\mathbf{0}, -1).$$
\end{enumerate}

\end{lemma}

\begin{proof}
By Proposition~\ref{p:ott}, $j \in S_{\omega}$ if and only if $j \in S_{\omega_-}$.  By \eqref{e:h2},
\[ \wh \LL \otimes \RR \cong H_2(\ol \cT; \RR) \oplus \bigoplus_{j \in S} \RR \wh \xi_j,\]
where $\wh \xi_j \in \wh \LL \otimes \QQ$ is as in \eqref{e:xi}. 
Proposition~\ref{p:ott} also implies that the fan for $\cT$ consists of the Type~\ref{i:bart1} cones of $\wh \Sigma_{\omega_-}$.  
Note also that for $1 \leq j \leq m$, $$\ol{\wh b_j} = \left(\ol b_j, -\phi_D(\ol b_j)\right).$$
The second part of Condition~\ref{a:as1} for $D$ then implies 
that for $j \in S$, $\ol{\wh b_j}$ is contained in one of the cones of $\cT$.  It follows by \eqref{e:xi}
that $\wh \xi_j = (\xi_j, 0)$, where $\xi_j \in \LL \otimes \QQ$ is the corresponding element for $\omega$.  Under the splittings
\[\bigoplus_{j \in S} \RR \wh \xi_j  \oplus H_2(\ol \cT; \RR)  = \wh \LL \otimes \RR = \LL \otimes \RR \oplus \RR = \bigoplus_{j \in S} \RR \xi_j \oplus H_2(\cT; \RR) \oplus \RR,\]
the summand $H_2(\ol \cT; \RR)$ on the left is identified with $H_2(\cT; \RR) \oplus  \RR$ on the right.
 
From the equation above we see that $\bigcap_{j \in S} \ker( \wh \xi_j) = \bigcap_{j \in S} \ker( \xi_j) \oplus \RR$.  This implies the second claim by \eqref{e:split2}.  
 
 To compare the extended ample cones, we use Proposition~\ref{p:ott}, which implies that the anticones are of the form
 \begin{enumerate}
 \item if $I \subset \{1, \ldots, m\}$ is an anticone of $\omega$, then $\wh I = I \cup \{\bee\}$ is an anticone of $\omega_-$; \label{i:t1}
 \item if $I \subset \{1, \ldots, m\}$ is an anticone of $\omega$ which is disjoint from $\{1, \ldots, k\}$, then $I ' =  I \cup \{\bff\}$ is an anticone of $\omega_-$.  \label{i:t2}
 \end{enumerate}
 
 First observe that 
 $$ \bigcap_{\wh I \text{ of type~\eqref{i:t1}}} \angle \wh I =C_\omega \times \RR_{\geq 0}(\mathbf{0}, -1).$$
 To see that this is equal to $C_{\omega_-}$ it suffices to show that the further intersection with $\angle I '$ for the anticones
 $I '$  of type~\eqref{i:t2} does not change the set.  This follows from the following claim: if $I \subset \{1, \ldots, m\}$ is an anticone of $\omega$ which is disjoint from $\{1, \ldots, k\}$, then  $\angle \wh I $ is contained in $ \angle I '$, where $\angle \wh I $ and $ \angle I '$  are the corresponding anticones of types~\eqref{i:t1} and~\eqref{i:t2} associated to $I$.  
 
 To prove the claim we show that 
 $(\mathbf{0}, -1) \in \angle I '.$
  This too follows from the convexity assumption of Condition~\ref{a:as1} for $D$.  Again using \eqref{e:span}, for $I$ disjoint from $\{1, \ldots, k\}$, 
 \[ (\mathbf{0}, -1) = \frac{1}{k-1}\left(\sum_{i \in I} c_i \wh D_i + \wh D_{\bff}\right) \in \angle I '.\]
 The claim follows and we conclude that 
  $$C_{\omega_-} = \left(\bigcap_{\wh I \text{ as in ~\eqref{i:t1}}} \angle \wh I \right)\bigcap \left(\bigcap_{I ' \text{ as in ~\eqref{i:t2}}} \angle I ' \right) = \bigcap_{\wh I \text{ as in ~\eqref{i:t1}}} \angle \wh I=C_\omega \times \RR_{\geq 0}(\mathbf{0}, -1).$$
  
  The comparison of ample cones follows from this and \eqref{e:ampcone}.  The comparison of Mori cones is then immediate.
  \end{proof}

\section{Crepant transformation conjecture}
In this section we recall the crepant transformation conjecture proven in \cite{CIJ}, and use it to prove analogous statements in compactly supported and narrow cohomology.  The results of this section will be applied specifically to $\wt \cT$ and $\ol \cT$.

\subsection{Wall crossing}\label{ss:wc}
We begin by recalling the general wall crossing setup of \cite{CIJ}.   To be notationally consistent with the specific setup we will use in Section~\ref{s:tots}, we will consider toric varieties $\cY_{\omega_+}$ and $\cY_{\omega_-}$ arising as GIT stack quotients by a torus $\wh K \cong (\CC^*)^{r+1}$ of dimension $r+1$.  The corresponding rank $r+1$ lattice $\Hom(\CC^*, \wh K)$ will be denoted by $\wh \LL$.  
Choose stability conditions $\omega_+, \omega_- \in \wh \LL^\vee \otimes \RR$ lying in cones $C_{\omega_+}$ and $C_{\omega_-}$ of maximal dimension which are separated by a codimension-one wall.  Denote by $\cY_{\omega_+}$ and $\cY_{\omega_-}$ the corresponding toric stacks.
Define 
$(\wh \bn, \wh \Sigma_{\omega_+}, \wh \beta, S_+)$ and $(\wh \bn, \wh \Sigma_{\omega_-}, \wh \beta, S_-)$ denote the extended stacky fans.

Define $W$ denote the hyperplane separating $C_{\omega_+}$ and $C_{\omega_-}$ and define
\[ \ol{C_W} := W \cap \ol C_{\omega_+} = W \cap \ol C_{\omega_-}.\]  
Let $e \in \wh \LL$ be a primitive generator of $W^\perp$.  We assume without loss of generality that $\omega_+ \cdot e >0$.  Under these assumptions, $\cY_{\omega_+}$ and $\cY_{\omega_-}$ are birational, via a common toric blow-up defined in \cite[Section~6.3.1]{CIJ} 
\begin{equation}\label{e:cbl}
\begin{tikzcd}
& \wh \cY \ar[dl, swap, "\pi_-"] \ar[dr, "\pi_+"] & \\  \cY_{\omega_-} && \cY_{\omega_+}.
\end{tikzcd}
\end{equation}  If we assume further that $(\sum_{1 \leq i \leq m} D_i) \cdot e = 0$ then $\cY_{\omega_+}$ and $\cY_{\omega_-}$ are $K$-equivalent.  This is the setting of \cite{CIJ}.

Consider the fan consisting of the cones $\ol C_{\omega_+}$, $\ol C_{\omega_-}$, and their faces.  
Let $\cM$ denote the corresponding toric variety.  
This is an open subset of the secondary toric variety associated to $\wh \Sigma_{\omega_+}$ and $\wh \Sigma_{\omega_-}$.  Let $P_{+/-}$ denote the torus fixed point of $\cM$ associated to the cone $\ol C_{\omega_{+/-}}$, and let $\cC$ denote the torus invariant curve between $P_+$ and $P_-$.  The correspondence given by the crepant transformation conjecture of \cite{CIJ} takes place on a formal neighborhood $\wh \cM$ of $\cC$ in $\cM$.  Following \cite{CIJ}, it is more convenient to work with a smooth cover of $\cM$.

Choose integral bases $\{p_1^+, \ldots, p_{r+1}^+\}$ and $\{p_1^-, \ldots, p_{r+1}^-\}$ of $\wh \LL^\vee$ such that 
\begin{itemize}
\item $p_i^{+/-} \in \ol C_{\omega_{+/-}}$ for $1 \leq i \leq r+1$;
\item $p_i^+ = p_i^- \in \ol{C_W}$ for $1 \leq i \leq r$.
\end{itemize}

For $\wh d \in \wh \LL$, let $ {\bm y}^{\wh d}$ denote the corresponding element of $\CC[\wh \LL]$.  We have inclusions 
\begin{align*} 
\CC[C_{\omega_+}^\vee \cap \wh \LL] \hookrightarrow \CC[y_1, \ldots, y_{r+1}]  &\qquad {\bm y}^{\wh d} \mapsto \prod_{i=1}^{r+1} y_i^{p_i^+ \cdot \wh d} 
\\
\CC[C_{\omega_-}^\vee \cap \wh \LL] \hookrightarrow \CC[\tilde y_1, \ldots, \tilde y_{r+1}] &\qquad {\bm y}^{\wh d} \mapsto \prod_{i=1}^{r+1} \tilde y_i^{p_i^- \cdot \wh d}.
\end{align*}
The coordinates are related by the change of variables 
\begin{align}\label{e:cov} \tilde y_i &= y_iy_{r+1}^{c_i} \text{ for  } 1\leq i\leq r; \\
\tilde y_{r+1} &= y_{r+1}^{-1},  \nonumber
\end{align}
where $c_i \in \ZZ$ are determined by the change of basis from $\{p_i^+\}$ to $\{p_i^-\}$.  Note that by assumption 
$p_{r+1}^+\cdot e = -p_{r+1}^-\cdot e = 1$.  Let $\cM ' $ denote the toric variety corresponding to the fan whose cones consist of $\wh C_{\omega_+} = \op{cone}\{p_1^+, \ldots, p_{r+1}^+\}$, $\wh C_{\omega_-} = \op{cone}\{p_1^-, \ldots, p_{r+1}^-\}$, and their faces.  There is a birational map $$\cM ' \to \cM$$ induced by the map of cones.

We will consider a modification of $\cM '$, whose sheaf of functions is analytic in the last coordinate ($r+1$) and formal in the other coordinates.
Let 
$\CC_{\omega_{+/-}}$ denote the analytic complex plane, with coordinates $y_{r+1}$ and $\tilde y_{r+1}$ respectively.  These glue to form an analytic $\PP^{1, \op{an}}$ via the change of variables $\tilde y_{r+1} = y_{r+1}^{-1}.$  Define the sheaves
\begin{align*} \cc O_{\wh \cU_+ '} &:= \cc O_{\CC_{+}}^{\op{an}} [[y_1, \ldots, y_{r}]]\\
\cc O_{\wh \cU_- '} &:= \cc O_{\CC_{-}}^{\op{an}} [[\tilde y_1, \ldots, \tilde y_{r}]]. \nonumber
\end{align*}
which also glue over $\PP^{1, \op{an}}$ via \eqref{e:cov}.  Let $\wh \cM '$ denote the corresponding ringed space, and let $\wh \cU_+ '$ and $\wh \cU_- '$ be the open sets given by $ y_{r+1} \neq 0$ and $\tilde y_{r+1} \neq 0$ respectively.

\begin{remark}\label{r:pim}
The setup above differs slightly from that given in \cite{CIJ}, where the authors define overlattices $\wt{\wh \LL}_+$ and $\wt{\wh \LL}_-$ of $\wh \LL$, and choose $\{p_1^{+/-}, \ldots, p_{r+1}^{+/-}\}$ to be an integral basis of $\wt{\wh \LL}_{+/-}^\vee$.  These coordinate are used to define a variety $\cM_{\op{reg}}$ \cite[Equation~5.10]{CIJ}, which is a smooth cover of $\cM$.  One can check that the bases may be chosen so that the map $\cM_{\op{reg}} \to \cM$ factors through $\cM ' \to \cM$.  
There are induced maps $\wh \cM_{\op{reg}} \to \wh \cM ' \to \wh \cM$ between the completions of the analytic spaces.

The overlattices $\wt{\wh \LL}_+$ and $\wt{\wh \LL}_-$ of $\wh \LL$ correspond to finite covers of $\cM$ and $\cM '$.
The deck transformations of these covers corresponds under the mirror theorem to \emph{Galois symmetries} of the quantum $D$-module \cite[Sections~2.2 and~2.3]{Iri}.

By \cite[Remarks~5.10 and~6.6]{CIJ}, the mirror theorem and crepant transformation conjecture of \cite{CIJ} are both compatible with this Galois action, consequently the results in this section are equally valid on $\wh \cM_{\op{reg}} $ and $ \wh \cM '$.
\end{remark}

 \begin{notation}\label{n:pim}
  By Lemma~\ref{l:amp}, for the particular case of $\cY_{\omega_-} = \ol \cT$ and $\cY_{\omega_+} = \wt \cT$, we may set $p_{r+1}^- \in \ol C_{\omega_-}$ to be $(\mathbf{0}, -1) =  D_{\bee}$.  We fix this choice in later sections.
  \end{notation}
  
\subsection{Mirror theorem and CTC}

We recall the main result of \cite{CIJ}.  In contrast to that paper we work non-equivariantly.

Let 
\[\mathfrak{c} := \prod_{\{j| \; D_j \cdot e \neq 0\}} (D_j \cdot e)^{D_j \cdot e}.\] 
Let $\wh \cM '_{+/-}$ denote $\cU_{+/-} ' \setminus\{ {\bm y}^e = \mathfrak{c}\}$.

\begin{definition}
Define the Fourier--Mukai transform 
$$ \mathbb{FM} := {\pi_+}_* \circ \pi_-^*: K^0(\cY_{\omega_-}) \to K^0(\cY_{\omega_+}),$$ wherer $\pi_{+/-}$ is as in \eqref{e:cbl}.  Because $\pi_+$ and $\pi_-$ are proper, we can also define 
$$\mathbb{FM_{\op{c}}} := {\pi_+}_* \circ \pi_-^*: K^0_{\op{c}}(\cY_{\omega_-}) \to K^0_{\op{c}}(\cY_{\omega_+}).$$
\end{definition}

\begin{theorem}\cite[Theorem~5.14]{CIJ}\label{t:mt}
There exists an open subset $\wh \cM_+^\circ \subset  \wh \cM '_+$ such that $\wh \cM '_+ \setminus \wh \cM_+^\circ$ is a discrete set and $ y_{r+1} = 0 \in \wh \cM_+^\circ$.  On $\wh \cM_+^\circ \times \op{Spec}(\CC[z])$ there exists
\begin{itemize}
\item A trivial $H^*_{\op{CR}}(\cY_{\omega_+})$-bundle  
\[\mathbf{F}^+:=H^*_{\op{CR}}(\cY_{\omega_+})\otimes\cc O_{\wh \cM_+^\circ}[z];\] 
\item A flat connection $\boldsymbol{\nabla}^{+}=d+z^{-1}\mathbf{A}({\bm y},z)$ on $\mathbf{F}^+$ with logarithmic singularities at $\{y_i = 0\}$ for $ 1 \leq i \leq r+1$;
\item A mirror map $\tau_+ :\wh \cM '_+ \to H^*_{\op{CR}}(\cY_{\omega_+})$, of the form
\begin{equation}\label{e:taup}\tau_+ = \sigma_+  + \tilde \tau_+ \text{ with } \tilde \tau_+ \in H^*_{\op{CR}}(\cY_{\omega_+}) \otimes\cc O_{\wh \cM_+^\circ}\end{equation} where $$\sigma_+ := \sum_{i=1}^{r+1} \theta^+ (p_i^+) \log y_i$$
and $\tilde \tau_+|_{y_1 = \cdots = y_{r+1} = 0} = 0$;
\end{itemize}
such that $\boldsymbol{\nabla}^{+}$ is the pullback of the quantum connection $\nabla^{\cY_{\omega_+}}$ of $\cY_{\omega_+}$ by $\tau_+$. 
\end{theorem}
The above is a $D$-module formulation of the mirror theorem from \cite{CCITMT}, itself a generalization of the work of \cite{G1, LLY}.  The connection $\boldsymbol{\nabla}^{+}$ is generated by the $I$-function $I_+({\bm y}, z)$ of \cite{CCITMT} after specializing the Novikov variables appropriately.  An analogous result holds for $\cY_{\omega_-}$.

Next we state the crepant transformation conjecture proven in \cite{CIJ}, which relates the connections $\boldsymbol{\nabla}^{+}$ and $\boldsymbol{\nabla}^{-}$.  The comparison may be made on 
$\wh \cM^\circ :=\wh \cM_+^\circ \setminus \{y_{r+1} = 0\} \cap \wh {\cM}_-^\circ \setminus \{\tilde y_{r+1} = 0\}$ and depends on a path of analytic continuation in $\PP^{1, \op{an}} \setminus \{{\bm y}^e = 0, \mathfrak{c}, \infty\}$ from a neighborhood of $\tilde y_{r+1} = 0$ to a neighborhood of $ y_{r+1} = 0$.
We choose a path $\gamma$ from $\log {\bm y}^e = - \infty$ to $\log {\bm y}^e = \infty$ such that the real part of $\log {\bm y}^e$ is always increasing, the imaginary part of $\log {\bm y}^e$ is $0$ when $\log |{\bm y}^e| \ll 0$ or  $\log |{\bm y}^e| \gg 0$, 
 and $\gamma$ contains the point $$\log |\mathfrak c| + \pi i \left(\sum_{j: D_j \cdot e > 0}D_j \cdot e > 0\right).$$

\begin{theorem}[crepant transformation conjecture \cite{CIJ}]\label{t:ctc}
There exists a gauge transformation
\[\Theta \in \op{Hom} \left( H^*_{\op{CR}}(\cY_{\omega_-}), H^*_{\op{CR}}(\cY_{\omega_+})
\right) \otimes \cc O_{\wh \cM^\circ}[z]\]
such that 
\begin{itemize}
\item the connections $\bnab^-$ and $\bnab^+$ are gauge equivalent via $\Theta$:
\[\bnab^+ \circ \Theta = \Theta \circ \bnab^-;\]
\item analytic continuation of flat sections is induced by a Fourier--Mukai transformation:
\[\Theta\left(\bs^{\cY_{\omega_-}}( \tau_-(\bm y), z)(E)\right) = \bs^{\cY_{\omega_+}}(\tau_+(\bm y), z)({\mathbb{ FM}}(E)).\]

\end{itemize}
\end{theorem}

\begin{remark}
In fact the theorem as stated in \cite{CIJ} takes place on an open subset of the universal cover of $\wh \cM '_+ \cap \wh \cM_-'$ (as that is where the $I$-functions are single-valued).  However the results are equally valid on $\wh \cM^\circ$, by \cite[Remark~6.6]{CIJ} (see also Remark~\ref{r:pim}).
\end{remark}

\subsection{Compactly supported CTC}

We recall from Section~\ref{ss:nqdm} the compactly supported quantum connection $\nabla^{\cY, \op{c}}$ and the compactly supported fundamental solution $L^{\cY, \op{c}}(\bt, z)$.

\begin{definition}
With $\wh \cM_{+/-}^\circ$ and the mirror map  $\tau_{+/-} :\wh \cM_{+/-} ' \to H^*_{\op{CR}}(\cY_{\omega_{+/-}})$ given as in Theorem~\ref{t:mt}, define the trivial $H^*_{\op{CR, c}}(\cY_{\omega_{+/-}})$-bundle over $\wh \cM_{+/-}^\circ \times \op{Spec}(\CC[z])$
\[\mathbf{F}^{{+/-},c} := H^*_{\op{CR, c}}(\cY_{\omega_{+/-}}) \otimes \cc O_{\wh \cM_{+/-}^\circ}[z].\]  Define the connection
\[\bnab^{{+/-},c} := \tau_{+/-}^*(\nabla^{\cY_{\omega_{+/-}}, \op{c}}).\]
\end{definition}
By Proposition~\ref{p:dualconn}, $\bnab^{+/-}$ and $\bnab^{{+/-},c}$ are dual with respect to $S^{\cY_{\omega_{+/-}}}$. 
\begin{definition} Define the $\cc O_{\wh \cM^\circ}[z]$-valued homomorphism
\[\Theta^{\op{c}} \in \op{Hom} \left( H^*_{\op{CR, c}}(\cY_{\omega_-}), H^*_{\op{CR, c}}(\cY_{\omega_+})
\right) \otimes \cc O_{\wh \cM^\circ}[z]\]
via the equation
\begin{equation}\label{d:thc} S^{\cY_{\omega_+}}( \Theta \alpha, \Theta^{\op{c}} \beta) = S^{\cY_{\omega_-}} (\alpha, \beta)\end{equation}
for all $\alpha \in H^*_{\op{CR}}(\cY_{\omega_-})$ and $\beta \in H^*_{\op{CR, c}}(\cY_{\omega_-})$.
\end{definition}

\begin{theorem}[compactly supported CTC]\label{t:cctc}
The connections $\bnab^{-, \op{c}}$ and $\bnab^{+, \op{c}}$ are gauge equivalent via $\Theta^{\op{c}}$:
\[\bnab^{+, \op{c}} \circ \Theta^{\op{c}} = \Theta^{\op{c}} \circ \bnab^{-, \op{c}}.\]
Furthermore, analytic continuation of flat sections is induced by a Fourier--Mukai transformation:
\[\Theta^{\op{c}}\left(\bs^{\cY_{\omega_-}, \op{c}}( \tau_-(\bm y), z)(E)\right) = \bs^{\cY_{\omega_+}, \op{c}}(\tau_+(\bm y), z)({\mathbb{ FM}}(E)),\]
for all $E \in K^0_{\op{c}}(\cY_{\omega_-})$.
\end{theorem}

\begin{proof}
For all $u(z) \in H^*_{\op{CR}}(\cX)[[\bq, \bt ']]((z^{-1}))$ and $v(z) \in H^*_{\op{CR, c}}(\cX)[[\bq, \bt ']]((z^{-1}))$, 
\begin{align*}
&S^{\cY_{\omega_+}} \left( \Theta \circ \bnab^{-} u(z), \Theta^{\op{c}} v(z)\right) + S^{\cY_{\omega_+}} \left( \Theta u(z), \Theta^{\op{c}} \circ \bnab^{-, \op{c}} v(z)\right) \\
=&S^{\cY_{\omega_-}} \left( \bnab^{-} u(z), v(z)\right) + S^{\cY_{\omega_-}} \left( u(z), \bnab^{-, \op{c}} v(z)\right) \\
=&
\partial_i S^{\cY_{\omega_-}} \left(u(z), v(z)\right) \\
=& \partial_i S^{\cY_{\omega_+}} \left(\Theta u(z), \Theta^{\op{c}} v(z)\right) \\
=& S^{\cY_{\omega_+}} \left( \bnab^{+} \circ \Theta u(z), \Theta^{\op{c}} v(z)\right) + S^{\cY_{\omega_+}} \left( \Theta u(z), \bnab^{+, \op{c}} \circ\Theta^{\op{c}}v(z)\right) \\
=& S^{\cY_{\omega_+}} \left( \Theta\circ \bnab^{-}  u(z), \Theta^{\op{c}} v(z)\right) + S^{\cY_{\omega_+}} \left( \Theta  u(z), \bnab^{+, \op{c}} \circ\Theta^{\op{c}}v(z)\right).
\end{align*}
The first and third equalities follow from the definition of $\Theta^{\op{c}}$.  The second and fourth follow from 
Proposition~\ref{p:dualconn}.
The fifth equality follows from
Theorem~\ref{t:ctc}.
Cancelling $S^{\cY_{\omega_+}} ( \Theta\circ \bnab^{-}  \alpha, \Theta^{\op{c}} \beta)$ from the top and bottom expression finishes the proof of the first statement.

To compare flat sections with the Fourier--Mukai transformation, consider $E \in K^0(\cY_{\omega_-})$ and $F \in K^0_{\op{c}}(\cY_{\omega_-})$.
\begin{align*}
&S^{\cY_{\omega_+}}\left(\bs^{\cY_{\omega_+}}( \tau_+(\bm y), z)(\mathbb{ FM}(E)) , \bs^{\cY_{\omega_+}, \op{c}}( \tau_+(\bm y), z)(\mathbb{ FM}_{\op{c}}(F))\right) \\
=& (-1)^{\dim (\cY_{\omega_+})}\chi(\mathbb{ FM}(E), \mathbb{ FM}_{\op{c}}(F)) \\
=& (-1)^{\dim (\cY_{\omega_-})}\chi(E, F ) \\
=& S^{\cY_{\omega_-}}\left(\bs^{\cY_{\omega_-}}( \tau_-(\bm y), z)(E), \bs^{\cY_{\omega_-}, \op{c}}( \tau_-(\bm y), z)(F)\right) \\
=& S^{\cY_{\omega_+}}\left(\Theta \bs^{\cY_{\omega_+}}( \tau_-(\bm y), z)(E), \Theta^{\op{c}} \bs^{\cY_{\omega_+}, \op{c}}( \tau_-(\bm y), z)(F)\right) \\
=& S^{\cY_{\omega_+}}\left( \bs^{\cY_{\omega_+}}( \tau_+(\bm y), z)(\mathbb{ FM}(E)), \Theta^{\op{c}} \bs^{\cY_{\omega_+}, \op{c}}( \tau_-(\bm y), z)(F)\right).
\end{align*}
As $\{\bs^{\cY_{\omega_+}}(\mathbb{ FM}(E)) | E \in K^0(\cY_{\omega_-})\}$ span the flat sections of $\nabla^+$ and the pairing $S^{\cY_{\omega_+}}$ is nondegenerate, we conclude that 
$$\Theta^{\op{c}} \bs^{\cY_{\omega_-}, \op{c}}( \tau_-(\bm y), z)(F) = \bs^{\cY_{\omega_+}, \op{c}}( \tau_+(\bm y), z)(\mathbb{ FM}(F)).
$$
\end{proof}

\subsection{Narrow CTC}
In this section we use the compactly supported crepant transformation conjecture to prove a corresponding statement in narrow cohomology.

\begin{definition}\label{d:narm}
With $\wh \cM_{+/-}^\circ$ and the mirror map  $\tau_{+/-} :\wh \cM_{+/-} ' \to H^*_{\op{CR}}(\cY_{\omega_{+/-}})$ given as in Theorem~\ref{t:mt}, define the trivial $H^*_{\op{CR, nar}}(\cY_{\omega_{+/-}})$-bundle over $\wh \cM_{+/-}^\circ \times \op{Spec}(\CC[z])$
\[\mathbf{F}^{{+/-}, \op{nar}} := H^*_{\op{CR, nar}}(\cY_{\omega_{+/-}}) \otimes \cc O_{\wh \cM_{+/-}^\circ}[z].\]  Define the connection
\[\bnab^{{+/-}, \op{nar}} := \tau_{+/-}^*(\nabla^{\cY_{\omega_{+/-}}, \op{nar}}).\]
\end{definition}

The following proposition was suggested to the second author by Iritani in private communication.
\begin{proposition}\label{p:byIri} Provided Assumption~\ref{a:ch2} holds for $\cY_{\omega_{+/-}}$, the following diagram commutes: 
\begin{equation}
\begin{tikzcd}[cramped, sep=small]
 K^0_{\op{c}}(\cY_{\omega_-}) \ar[rr, rightarrow, "\varphi_{K^0}"] \ar[rd, "\mathbb{FM}_{\op c}"] \ar[dd]& & K^0(\cY_{\omega_-})  \ar[rd, "\mathbb{FM}"] \ar[dd] &  \\
 & K^0_{\op{c}}(\cY_{\omega_+}) \ar[rr, rightarrow, crossing over, near start, "\varphi_{K^0}"] & &  K^0(\cY_{\omega_+}) \ar[dd]\\
  \op{ker}(\nabla^{-, \op{c}}) \ar[rr, rightarrow, near end, "\varphi"]  
  \ar[rd, rightarrow, "\Theta^{\op{c}}"] & & \op{ker}(\nabla^{-})  \ar[rd, rightarrow, "\Theta"]  &  \\
 & \op{ker}(\nabla^{+, \op{c}}) \ar[from = uu, crossing over] \ar[rr, rightarrow, "\varphi" ']   & &  \op{ker}(\nabla^{+})
\end{tikzcd}
\end{equation}
where the vertical arrows are $\bs^{\cY_{\omega_{+/-}}, \op{c}}(\tau_{+/-}(\bm y), z)( - )$ and $\bs^{\cY_{\omega_{+/-}}}(\tau_{+/-}(\bm y), z)( - )$.

In particular,
the map $\Theta \in \op{Hom} \left( H^*_{\op{CR}}(\cY_{\omega_-}), H^*_{\op{CR}}(\cY_{\omega_+})
\right) \otimes \cc O_{\wh \cM^\circ}[z]$ sends $H^*_{\op{CR, nar}}(\cY_{\omega_-})$ to $H^*_{\op{CR, nar}}(\cY_{\omega_+})$.
\end{proposition}

\begin{proof}
Commutativity of the top square is immediate from the definition of $K^0_{\op{c}}$.  Commutativity of the front and back squares follows from \eqref{e:phiK}
and that
the fundamental solutions satisfy 
\[\varphi \circ L^{+/-, c} = L^{+/-} \circ \varphi.\]
Commutativity of the left and right square are Theorems~\ref{t:cctc} and~\ref{t:ctc} respectively.  Commutativity of the bottom square then holds if the maps $\bs^{\cY_{\omega_{+/-}}, \op{c}}$ are surjective.  This holds because Assumption~\ref{a:ch2} holds for the total space of a vector bundle on a smooth toric stack.
\end{proof}

\begin{definition}\label{d:thetanar}
Define $\Theta^{\op{nar}} \in \op{Hom} \left( H^*_{\op{CR, nar}}(\cY_{\omega_-}), H^*_{\op{CR, nar}}(\cY_{\omega_+})
\right) \otimes \cc O_{\wh \cM^\circ}[z]$ to be the restriction of $\Theta$ to $H^*_{\op{CR, nar}}(\cY_{\omega_-})$.
\end{definition}

\begin{theorem}[narrow CTC]\label{t:nctc}
The transformation $\Theta^{\op{nar}}$ 
satisfies the following:
\begin{itemize}
\item the connections $\bnab^{-, \op{nar}}$ and $\bnab^{+, \op{nar}}$ are gauge equivalent via $\Theta^{\op{nar}}$:
\[\bnab^{+, \op{nar}} \circ \Theta^{\op{nar}} = \Theta^{\op{nar}} \circ \bnab^{-, \op{nar}};\]
\item analytic continuation of flat sections is induced by a Fourier--Mukai transformation:
\[\Theta^{\op{nar}}\left(\bs^{\cY_{\omega_{+/-}}, \op{nar}}( \tau_-(\bm y), z)(E)\right) = \bs^{\cY_{\omega_{+/-}}, \op{nar}}(\tau_+(\bm y), z)(\mathbb{ FM}(E));\] 
\item The pairing is preserved:
\[S^{\cY_{\omega_+}, \op{nar}}( \Theta^{\op{nar}} \alpha, \Theta^{\op{nar}} \beta) = S^{\cY_{\omega_-}, \op{nar}} (\alpha, \beta).\]
\end{itemize}
\end{theorem}
\begin{proof}
The first point and second point are automatic from the fact that  $\bnab^{+/-, \op{nar}}$, $\Theta^{\op{nar}}$, and $\bs^{\cY_{\omega_{+/-}}, \op{nar}}$ are restrictions of $\bnab^{+/-}$, $\Theta$ and $\bs$ respectively.  The third point is due to the fact that for $\alpha \in H^*_{\op{CR, nar}}(\cY_{\omega_-})$ and $\beta \in H^*_{\op{CR, c}}(\cY_{\omega_-})$,
\begin{align*} S^{\cY_{\omega_+}, \op{nar}}( \Theta^{\op{nar}} \alpha, \Theta^{\op{nar}} \circ \varphi (\beta)) =&  S^{\cY_{\omega_+}, \op{nar}}( \Theta \alpha,  \Theta \circ \varphi( \beta)) \\
=&  S^{\cY_{\omega_+}, \op{nar}}( \Theta \alpha, \varphi \circ \Theta^{\op{c}} \beta) \\
=&S^{\cY_{\omega_+}}( \Theta \alpha, \Theta^{\op{c}} \beta) \\
=&S^{\cY_{\omega_-}} (\alpha, \beta) \\
=&S^{\cY_{\omega_-}, \op{nar}} (\alpha, \varphi (\beta)).
\end{align*}
\end{proof}

\section{$D$-module of the partial compactification}\label{s:Dcomp}
In this section we identify a subquotient of a certain restriction of $\nabla^{\ol \cT, \op{nar}}$ with $\nabla^{\cT, \op{nar}}$.  The results rely on the specific geometry of $\ol \cT$.

\subsection{Restricting the Novikov variable}

\begin{notation}
Lemma~\ref{l:amp} allows us to express any $\wh d \in NE(\ol \cT)$, as $d + d'$ where $d \in NE(\cT) $ and $d'$ is a multiple of $(\mathbf 0, -1)$.  We can then decompose the Novikov ring of $\ol \cT$ into the part coming from $\cT$ and the part coming from $u_\bee$.
Let $Q ' = Q^{(\mathbf{0}, -1)} $ denote the Novikov variable corresponding to the last factor of $\op{NE}(\ol \cT) = \op{NE}(\cT) \oplus \ZZ_{\geq 0}(\mathbf{0}, -1).$ 
  Given $\wh d \in NE(\ol \cT)$, we write
\[\wh Q^{\wh d} = Q^d (Q ')^{d '},\]
where $Q^d$ is a monomial in the Novikov ring for $\cT$.
\end{notation}

We
will focus on the subspace $u_\bff H^*_{\op{CR}}(\ol \cT)$, which may be alternatively described as 
$$\op{im}\left(\ol j_*: H^*_{\op{CR}}(\cX) \to H^*_{\op{CR}}(\ol \cT)\right),$$
where $\ol j: \cX \to \ol \cT$ is the closed embedding obtained by composing $j:\cX \to \cT$ with the inclusion $\cT \subset \ol \cT$.

\begin{lemma}\label{l:subD}
After  restricting to $Q' = 0$, the quantum $D$-module $\nabla^{\ol \cT, \op{nar}}$ preserves the subspace $u_\bff  H^*(\ol \cT)$.  Furthermore, 
the map $i^*: u_\bff  H^*_{\op{CR}}( \ol \cT) \to H^*_{\op{CR, nar}}(\cT)$ is compatible with quantum $D$-modules after this restriction:
\begin{equation}\label{e:Lpull} i^* \left( L^{\ol \cT, \op{nar}}(\wh Q, \bt , z)|_{Q' = 0}z^{-\op{Gr}} z^{\rho(\ol \cT)} \alpha\right) =  L^{\cT, \op{nar}}( Q, i^* \bt , z)z^{-\op{Gr}} z^{\rho(\cT)} i^*(\alpha )\end{equation}
whenever $\alpha \in u_{\bff} H^*_{\op{CR}}(\ol \cT)$.
\end{lemma}

\begin{proof}

Let $f: (\cC, p_1, \ldots, p_n) \to \ol \cT$ be a stable map of degree $d = (d,0)$ in  $H_2(\cT; \QQ)  \times \{0\} \subset H_2(\ol \cT; \QQ)$. By \cite{Cox} and \cite[Section~2.2]{CiKi}, the map $f$ is determined by $r+1$ line bundles (corresponding to the factors of $\wh K$), and for each character $\wh D_i$, a section $s_i$ of the corresponding line bundle $\sL_i \to \cC$ (here $\sL_i = f^* \cc L_i$ from \eqref{e:canline}).  
By the condition on $d$, the line bundle  $\sL_{\bee}$ 
is trivial and the section $s_{\bee}$ is constant.  Therefore $f$ maps entirely to the locus $\{x_{\bee} \neq 0\} = \cT$ or $\{x_{\bee} = 0\} = \ol \cT \setminus \cT.$  
This implies in particular that 
we have the following cartesian diagram:
\begin{equation}\label{e:fibev}
\begin{tikzcd}
\sMbar_{g,n}(\cT, d) \ar[r] \ar[d, "\op{ev}_{i}"] & \sMbar_{g,n}( \ol \cT, d) \ar[d, "\op{ev}_{i}"] \\
\ol I\cT \ar[r, "i"]  &\ol I\, \ol \cT.
\end{tikzcd}
\end{equation}
(By a slight abuse of notation we will denote both the maps $\cT \to \ol \cT$ and $\ol I \cT  \to \ol I\, \ol \cT$ by $i$.)

We claim further that $f$ maps entirely to $\cX$ or $\ol \cT \setminus \cX$.  From above, we can split the argument into two cases depending on whether $f$ maps to $\cT$ or to $\ol \cT \setminus \cT$.
If $f$ maps entirely to $\ol \cT \setminus \cT$ then $f$ maps to $\ol \cT \setminus \cX$.  If $f$ maps entirely to $\cT$ then $f$ may be described by a map $\ol f$ to $\cX$ together with a section $\ol s$ of $\ol f^*(\cc O_\cX(-D))$.  By assumption $D$ is nef and therefore $\ol f^*(\cc O_\cX(D))$ has non-negative degree.  Thus the degree of $\ol f^*(\cc O_\cX(-D))$ is non-positive which forces the section $\ol s$ to be constant.  This shows the claim.
We therefore have a second cartesian diagram
\begin{equation}\label{e:fibev2}
\begin{tikzcd}
\sMbar_{g,n}(\ol \cT\setminus \cX, d) \ar[r, "\wt \mu"] \ar[d, "\op{ev}_{i}"] & \sMbar_{g,n}( \ol \cT, d) \ar[d, "\op{ev}_{i}"] \\
\ol I\, \ol \cT\setminus \bar I\cX \ar[r, "\mu"]  &\ol I\, \ol \cT.
\end{tikzcd}
\end{equation}

Given $\alpha \in u_{\bff} H^*_{\op{CR}}(\ol \cT) = \op{im}(\ol j_*)$, we note that $\alpha \in \op{ker}(\mu^*)$ and therefore 
$$\op{ev}_i^* \alpha \in \op{ker}(\wt \mu^*)$$ by \eqref{e:fibev2}.
From  \cite[Proposition~1.7]{Fu}  we have
$${\op{ev}_{i}}_* \wt \mu^* = \mu^* {\op{ev}_{i}}_*,$$
therefore
$$\mu^* {\op{ev}_{n+2}}_* \left( \left[\sMbar_{g,n+2}( \ol \cT, d)\right]^{\op{vir}} \cap \psi_1^a \cup \op{ev}_1^*(\alpha) \cup \prod_{k=2}^{n+1} \op{ev}_k^*(\bt)\right) = 0$$
for $\alpha \in u_{\bff} H^*_{\op{CR}}(\ol \cT)$.  Working term by term on the expression~\eqref{e:Lalt}, we conclude that $L^{\ol \cT, \op{nar}}(\wh Q, \bt , z)|_{Q' = 0} \alpha$ lies in $\op{ker}(\mu^*) = \op{im}(\ol j_*) = u_{\bff} H^*_{\op{CR}}(\ol \cT)$.  This shows that 
$L^{\ol \cT, \op{nar}}(\wh Q, \bt , z)|_{Q' = 0}$ preserves the subspace $u_{\bff} H^*_{\op{CR}}(\ol \cT)$.

Next we prove \eqref{e:Lpull}.  It suffices to show
\begin{equation}\label{e:Lpull2} i^* L^{\ol \cT, \op{nar}}(\wh Q, \bt , z)|_{Q' = 0} \alpha = L^{\cT, \op{nar}}(Q,i^* \bt , z) i^* \alpha,\end{equation}
We make use of \eqref{e:fibev}.
The perfect obstruction theory on  $\sMbar_{g,n}(\cT, d)$ is the restriction of the perfect obstruction theory on $\sMbar_{g,n}( \ol \cT, d)$, and so they are compatible in the sense of \cite[Definition~5.8]{BF}.  By functoriality \cite[Proposition~5.10]{BF}, $$i^! \left[\sMbar_{g,n}( \ol \cT, d)\right]^{\op{vir}} = \left[\sMbar_{g,n}(\cT, d)\right]^{\op{vir}}.$$
From this we observe that
\begin{align*} &i^* {\op{ev}_{n+2}}_* \left( \left[\sMbar_{g,n+2}( \ol \cT, d)\right]^{\op{vir}} \cap \psi_1^a \cup \op{ev}_1^*(\alpha) \cup \prod_{k=2}^{n+1} \op{ev}_k^*(\bt)\right) \\ = & 
{\op{ev}_{n+2}}_* i^! \left( \left[\sMbar_{g,n+2}( \ol \cT, d)\right]^{\op{vir}} \cap \psi_1^a \cup \op{ev}_1^*(\alpha) \cup \prod_{k=2}^{n+1} \op{ev}_k^*(\bt)\right) \\ = & 
{\op{ev}_{n+2}}_* \left( i^!\left(\left[\sMbar_{g,n+2}( \ol \cT, d)\right]^{\op{vir}} \right) \cap \psi_1^a \cup \op{ev}_1^*(i^* \alpha) \cup \prod_{k=2}^{n+1} \op{ev}_k^*(i^* \bt)\right) \\ = & 
{\op{ev}_{n+2}}_* \left( \left[\sMbar_{g,n+2}(\cT, d)\right]^{\op{vir}} \cap \psi_1^a \cup \op{ev}_1^*(i^* \alpha) \cup \prod_{k=2}^{n+1} \op{ev}_k^*(i^* \bt)\right).\end{align*}
This gives a term-by-term identification of the left and right sides of \eqref{e:Lpull2}.
\end{proof}

\begin{notation}
Following Notation~\ref{n:basis}, by decomposing $H^*_{\op{CR}}(\ol \cT)$ with respect to twisted sectors and degree, there is a canonical way to write a point $\bt \in H^*_{\op{CR}}(\ol \cT)$ as $\bt' + \bt_{(0,2)}$ where $\bt_{(0,2)} \in H^2(\ol \cT)$. 
 By Part~\ref{i:sec} of Lemma~\ref{l:amp}, the point $\bt_{(0,2)}$ may be canonically written as  $\bt_{(0,2)} ' + s\cdot u_{\bee}$ where $\bt_{(0,2)} ' \in H^2(\cT)$ and $s$ is a scalar.  We note  that $u_{\bee}$ is identified with $(\mathbf 0, -1)$ under \eqref{e:split2}.  For $d \in H_2(\cT; \QQ)$, let $\bq^d = e^{\langle \bt_{(0,2)} ', d\rangle^T}$. Let $q'$ denote $e^s$.  Then for $\wh d = d + d' \in H_2(\ol \cT; \QQ)$, $$e^{\langle \bt_{(0,2)}, \wh d\;\rangle^{\ol \cT}} = \bq^d \cdot (q')^{d'}.$$
\end{notation}
We consider the monodromy of $QDM_{\op{nar}}(\ol \cT)$ around $q'$.

\begin{lemma}\label{l:moninvt}
 The monodromy invariant part of the narrow quantum $D$-module of $\ol \cT$ around $q' = 0$ is spanned by the flat sections  $L^{\ol \cT, \op{nar}}(\bt, z)z^{- \op{Gr}}z^{\rho(\ol \cT)} \alpha$ for $\alpha$ ranging over
$$ \bigoplus_{\nu \in \wh \KK_{\omega_-}^{\op{int}}/\wh \LL} \op{ker}\left( u_\bee \cdot - : H^*_{\op{nar}}(\ol \cT_\nu) \to H^*_{\op{nar}}(\ol \cT_\nu)\right).$$
 In particular, it contains the flat sections 
 $$\{L^{\ol \cT, \op{nar}}(\bt, z) \alpha | \alpha \in u_{\bff} H^*_{\op{CR}}(\ol \cT)\}.$$  
\end{lemma}

\begin{proof}
As described in  \cite[Sections~2.2 and~2.3]{Iri}, the monodromy transformation of flat sections is given by the Galois action defined in \cite[Proposition~2.3]{Iri}.  The monodromy around $q' = 0$ corresponds to the action of $\cc L_{\bee}$: for $\alpha \in H^*(\ol \cT_\nu)$,
\begin{equation}\label{e:mond} dG(\cc L_{\bee}) L^{\ol \cT, \op{nar}}(G(\cc L_{\bee})^{-1} \bt, z)  z^{- \op{Gr}}z^{\rho(\ol \cT)}\alpha = L^{\ol \cT, \op{nar}}(\bt, z)  z^{- \op{Gr}}z^{\rho(\ol \cT)}
e^{2 \pi i u_{\bee}} e^{2 \pi i  m_\nu(\cc L_{\bee})} 
\alpha,\end{equation} 
where $\cc L_{\bee}$ is as in \eqref{e:canline} and $m_\nu(\cc L_{\bee})$ is the weight of the generator of the generic isotropy group of $\ol \cT_\nu$ on the line bundle $\cc L_{\bee}$.

We prove the second claim first.  For $\alpha \in u_{\bff} H^*_{\op{CR}}(\ol \cT)$, $\alpha \cdot u_{\bee} = 0$ and by Lemma~\ref{l:narsplit},
$\alpha$ is supported on the twisted sectors for which $m_\nu(\cc L_{\bee}) = 0$.  In this case the section $L^{\ol \cT, \op{nar}}(\bt, z) z^{- \op{Gr}}z^{\rho(\ol \cT)} \alpha$ is invariant under the monodromy transformation by \eqref{e:mond}.

For the first statement, choose a basis $\{\alpha_i\}_{i \in I}$ of $H^*_{\op{CR, nar}}(\ol \cT)((z^{-1}))$ such that each $\alpha_i$ is homogeneous and supported on a single twisted sector $\ol \cT_{\nu_i}$.  Then a section 
$L^{\ol \cT, \op{nar}}(\bt, z)  z^{- \op{Gr}}z^{\rho(\ol \cT)}\sum_{i \in I} c_i \alpha_i$ is monodromy invariant if and only if 
$$L^{\ol \cT, \op{nar}}(\bt, z)  z^{- \op{Gr}}z^{\rho(\ol \cT)} \sum_{ i \in I} c_i \alpha_i = L^{\ol \cT, \op{nar}}(\bt, z)   z^{- \op{Gr}}z^{\rho(\ol \cT)}\sum_{ i \in I} c_i e^{2 \pi i u_{\bee}}
e^{2 \pi i  m_{\nu_i}(\cc L_{\bee})} \alpha_i,$$ or equivalently, if
 $$\sum_{ i \in I} c_i \alpha_i =  \sum_{ i \in I} c_i e^{2 \pi i u_{\bee}}
e^{2 \pi i  m_{\nu_i}(\cc L_{\bee})} \alpha_i.$$
As multiplication by $e^{2 \pi i u_{\bee}}
e^{2 \pi i  m_{\nu}(\cc L_{\bee})} $ preserves $H^*(\ol \cT_\nu)$ for each twisted sector, we can work with each summand individually.  
Let $I_\nu \subset I$ be the subset such that $\{\alpha_i\}_{i \in I_\nu}$ is a basis for $H^*(\ol \cT_\nu)$.
By considering the top degree nonzero homogeneous term of $\sum_{ i \in I_\nu} c_i \alpha_i $, we see that it must lie in  $\op{ker}(u_{\bee} \cdot -)$, and that 
  $ m_{\nu_i}(\cc L_{\bee}) $ must be zero if $\sum_{ i \in I_\nu} c_i \alpha_i $ is nonzero.  Therefore for $\nu \in \wh \KK_{\omega_-}^{\op{frac}}/\wh \LL$, $c_i = 0$ for $i \in I_\nu$.
For $\nu \in \wh \KK_{\omega_-}^{\op{int}}/ \wh \LL$, 
we use descending induction on the top degree of the left hand side to see that
$$\sum_{ i \in I} c_i \alpha_i  \in \op{ker}(u_{\bee} \cdot -).$$
\end{proof}

The previous two lemmas together yield the following.
\begin{corollary}\label{c:monod}  
The monodromy invariant part of $QDM_{\op{nar}}(\ol \cT)$, when restricted to $q' = 0$, contains  a sub-$D$-module 
$QDM_{\op{nar}}(\ol \cT)|_{u_\bff H^*_{\op{CR}}(\ol \cT)}$ which maps surjectively via $i^*$ onto $QDM_{\op{nar}}(\cT)$.

This map of $D$-modules is compatible with the pairings.  It relates the integral structures as follows:  for $E \in K^0_\cX(\ol \cT)$, 
$$ i^*\left(\bs^{\ol \cT, \op{nar}}(\bt, z)(E)|_{q' = 0} \right)= \bs^{\cT, \op{nar}}( \bt, z)(i^*E).$$  
\end{corollary}

\begin{proof}
By   Lemma~\ref{l:moninvt}, the monodromy-invariant flat sections of $\nabla^{\ol \cT}$ contain the span of $L^{\ol \cT, \op{nar}}(\wh Q, \bt, z)|_{\wh Q=1}z^{- \op{Gr}}z^{\rho(\ol \cT)} \alpha$ for $\alpha \in u_{\bff} H^*_{\op{CR}}(\ol \cT)$.
 Equation~\eqref{e:Ldivred} implies that  
\[L^{\ol \cT, \op{nar}}(\wh Q, \bt, z) =  L^{\ol \cT, \op{nar}}(\wh Q, \bt ', z) e^{- \bt_{(0,2)} ' /z} (q')^{-u_{\bee}/z}\;\vline_{\;\begin{subarray}{l} Q =  Q \cdot \bq \\ Q' = Q' \cdot q' \end{subarray}}.\]
Thus for $\alpha \in u_{\bff} H^*_{\op{CR}}(\ol \cT)$, restricting $L^{\ol \cT, \op{nar}}(\wh Q, \bt, z) \alpha$ to $q' = 0$ is the same as restricting the Novikov variable $Q'$ to zero.  Then by Lemma~\ref{l:subD},  
\begin{equation}\label{e:Lrest}i^*\left(L^{\ol \cT, \op{nar}}(\bt , z)|_{q' = 0} z^{-\op{Gr}} z^{\rho(\ol \cT)}\alpha \right) = L^{\cT, \op{nar}}(i^* \bt , z) z^{-\op{Gr}} z^{\rho(\cT)} i^*(\alpha).
\end{equation}
This shows that the pullback $i^*$ maps $\nabla^{\ol \cT, \op{nar}}|_{u_\bff H^*_{\op{CR}}(\ol \cT)}$ to $\nabla^{\cT, \op{nar}}$.

To see that the pairings match, we apply the projection formula \cite[Proposition~6.15]{BT}
twice:
\begin{align*}
\langle \ol j_* \alpha, \ol j_* \beta \rangle^{\ol \cT, \op{nar}} &= \langle  \alpha,  \beta \cup u_\bff \rangle^{\ol \cX} \\
& = \langle j_* \alpha, j_* \beta \rangle^{\cT, \op{nar}} \\
& = \langle i^* \ol j_* \alpha, i^* \ol j_* \beta \rangle^{\cT, \op{nar}}.
\end{align*}
We note that the pairing $\langle - , - \rangle^{\ol \cT, \op{nar}}$ may become degenerate when restricted to $u_\bff H^*_{\op{CR}}(\ol \cT)$.

For $E \in K^0_\cX(\ol \cT)$, $\wt{\op{ch}}(E)$ is entirely supported in $u_{\bff} H^*_{\op{CR}}(\ol \cT)$.
The last statement follows from \eqref{e:Lrest} and the definition of $\bs$.
\end{proof}

\subsection{Interaction with the mirror theorem}

Recall the connection $\bnab^{\ol \cT, \op{nar}}$ on $\wh {\cM}_{-}^\circ$.  By Definition~\ref{d:narm}, this is the pullback of $\nabla^{\ol \cT, \op{nar}}$ via a mirror map 
$$\tau_{\ol \cT}:=\tau_- :\wh \cM_{-} ' \to H^*_{\op{CR}}(\ol \cT).$$  We consider the monodromy invariant submodule around $\tilde y_{r+1} = 0$.

\begin{proposition}\label{p:moninv}
The monodromy invariant part of $\bnab^{\ol \cT, \op{nar}}$, when restricted to $\tilde y_{r+1} = 0$, contains  a submodule which maps surjectively via $i^*$ onto $(i^* \circ \tau_{\ol \cT})^*(\nabla^{\cT, \op{nar}})$.

This identification is compatible with the pairings and integral structures as in Corollary~\ref{c:monod}.  \end{proposition}
\begin{proof}
Consider the pullback  $L^{\ol \cT, \op{nar}}(\tau_{\ol \cT}({\bm y}),z)$ of the fundamental solution matrix via the mirror map $\tau_{\ol \cT}({\bm y})$.
By Theorem~\ref{t:mt}, the mirror map is given by $$\tau_{\ol \cT}({\bm y}) = \sum_{i=1}^{r+1} \log(\tilde y_i)\theta^-(p_i^-) + \wt \tau_{\ol \cT}({\bm y}).$$  
From Notation~\ref{n:pim} we observe that $\theta^-(p_{r+1}^-) = \theta^-(D_{\bee}) = u_\bee$.
Thus we are in the setting of Corollary~\ref{c:monod}, with $\tilde y_{r+1}$ playing the role of $q'$.  The result follows.
\end{proof}

\section{Main theorem}
In this section we prove the main theorem, relating the narrow quantum $D$-modules of $\cZ$ and $\wt \cZ$.
We recall the relations between the various stacks described in Sections~\ref{s:ts} and~\ref{s:tots}:
\begin{equation}\label{d:hache}
\begin{tikzcd}
&\cT \ar[r, hookrightarrow,"i"] \ar[d, " \pi"] & \ol \cT \ar[r,  dashed, leftrightarrow] & \wt \cT \ar[d, swap, "\wt \pi"]& 
\\
\cZ \ar[r, "k"]&\ar[u, bend left = 25, " j"]  \cX \ar[ur, swap, "\ol j"]& & \wt \cX \ar[u, bend right = 25, swap, "\wt j"] & \wt \cZ \ar[l, swap, "\wt k"]
\end{tikzcd}
\end{equation}
Consider the following diagram, which combines the maps appearing in quantum Serre duality, the narrow crepant transformation conjecture, and  the previous section:
\[
\begin{tikzcd}[column sep = small]
H^*_{\op{CR}}(\ol \cT) \ar[d, swap, "i^*"] &\wh \cM^\circ_- \ar[l, "\tau_{\ol \cT}"]  
\ar[r, leftarrow] &\wh \cM^\circ \ar[r]&\wh \cM^\circ_+ \ar[r, swap, "\tau_{\wt \cT}"]  \ar[ddr, swap, "\tau_{\wt \cZ}"]
& H^*_{\op{CR}}(\wt \cT) \ar[dd,  "\ol f_{\wt \cT}"]\\
H^*_{\op{CR}}(\cT) \ar[d, swap, "\ol f_{ T}"]  & \wh \cM^\circ_-|_{\tilde y_{r+1} = 0} \ar[u, hookrightarrow] 
\ar[dl,  "\tau_\cZ"] & & \\
H^*_{\op{CR, amb}}(\cZ)& &&& H^*_{\op{CR, amb}}( \wt \cZ),
\end{tikzcd}
\]
where here $\tau_{ \cZ}$ and $\tau_{\wt \cZ}$ are defined to be the compositions.

\begin{notation}
We recall the following maps: \begin{itemize}
\item
Let
 $\Theta^{\op{nar}} \in \op{Hom} \left( H^*_{\op{CR, nar}}(\ol \cT), H^*_{\op{CR, nar}}(\wt \cT)
\right) \otimes \cc O_{\wh \cM^\circ}[z]$ be the gauge equivalence as given in Definition~\ref{d:thetanar} with respect to the wall crossing between $\ol \cT$ and $\wt \cT$.
\item
Let $ \bar \Delta_\cT: H^*_{\op{CR, nar}}(\cT) \to H^*_{\op{CR, amb}}(\cZ)$ denote the quantum Serre duality map for $\cT$ (Definition~\ref{d:qsd}).
\item
Let $ \bar  \Delta_{\wt \cT}: H^*_{\op{CR, nar}}(\wt \cT) \to H^*_{\op{CR, amb}}(\wt \cZ)$ denote the quantum Serre duality map for $\wt \cT$.
\end{itemize}
\end{notation}

Our main theorem states, roughly, that the quantum $D$-module $QDM_{\op{amb}}(\wt \cZ)$, together with its associated integral structure, can be analytically continued to a neighborhood of $ y_{r+1} = \infty$ ($\tilde y_{r+1} = 0$), and that the restriction of the monodromy invariant part to $\tilde y_{r+1} = 0$ has a sub-quotient which is gauge-equivalent to $QDM_{\op{amb}}(\cZ)$.  As a result, the quantum $D$-module for $\cZ$ is determined by the quantum $D$-module for $\wt \cZ$ after pullback by a mirror map.
The theorem follows almost immediately from the results of previous sections, namely:
\begin{enumerate} \item Theorem~\ref{t:nctc}, the (narrow) crepant transformation conjecture, which relates $QDM_{\op{nar}}(\ol \cT)$ to $QDM_{\op{nar}}(\wt \cT)$; 
\item Proposition~\ref{p:moninv}, which relates $QDM_{\op{nar}}(\ol \cT)$ to $QDM_{\op{nar}}(\cT)$; and 
\item Theorem~\ref{t:qsd}, quantum Serre duality, which relates $QDM_{\op{nar}}(\cT)$ to $QDM_{\op{amb}}(\cZ)$ and $QDM_{\op{nar}}(\wt \cT)$ to $QDM_{\op{amb}}(\wt \cZ)$.  
\end{enumerate}

To formulate a precise statement, we make use of the ambient quantum $D$-module of $\ol \cT$.
Define the quantum $D$-module $\bnab$ to be the pullback $\tau_{\ol \cT}^* \nabla^{\ol \cT, \op{nar}}$, together with the pulled-back integral structure and pairing.  
\begin{theorem}\label{t:mt}
The quantum $D$-module $\bnab$ on $ \wh \cM^\circ$ satisfies the following:
\begin{itemize}
\item In a neighborhood of $y_{r+1}= 0$, $\bnab$ is gauge-equivalent to $\tau_{\wt \cZ}^* \nabla^{\wt \cZ, \op{amb}}$ via the transformation $\Theta^{\wt \cZ} :=  \bar \Delta_{\wt \cT}\circ \Theta^{\op{nar}}$;
\item The monodromy invariant part of $\bnab$ around $\tilde y_{r+1} = 0$, when restricted to $ \tilde y_{r+1} = 0$, contains a sub-module which 
maps surjectively 
 to $\tau_{\cZ}^* \nabla^{\cZ, \op{amb}}$ via the gauge transformation $\Theta^\cZ:=  \bar \Delta_\cT \circ i^*$;
\item Both transformations preserves the pairing:
for $\alpha, \beta \in u_\bff H^*_{CR}(\ol \cT)$,
$$ S^{\ol \cT, \op{nar}}( \alpha, \beta) = S^{\cZ, \op{amb}}(\Theta^{\cZ} \alpha, \Theta^{\cZ} \beta);$$
for $\alpha, \beta \in H^*_{\op{CR, nar}}(\ol \cT)$
$$ S^{\ol \cT, \op{nar}}(  \alpha,  \beta) = S^{\wt \cZ, \op{amb}}(\Theta^{\wt \cZ}  \alpha, \Theta^{\wt \cZ}  \beta).$$
In particular,
for $\alpha, \beta \in u_\bff H^*_{CR}(\ol \cT)$,
$$S^{\cZ, \op{amb}}(\Theta^{\cZ}   \alpha, \Theta^{\cZ}   \beta) = S^{\wt \cZ, \op{amb}}( \Theta^{\wt \cZ}\alpha, \Theta^{\wt \cZ} \beta).$$ 
\item  
The integral lattice of $\tau_{\cZ}^* \nabla^{\cZ, \op{amb}}$ is obtained from the restriction of the monodromy invariant sublattice of the integral lattice of $\tau_{\wt \cZ}^* \nabla^{\wt \cZ, \op{amb}}$ to $\tilde y_{r+1} = 0$.
More precisely, for all $E \in K^0(\cX)$, 
$$\bs^{\cZ, \op{amb}}(\tau_\cZ(\bm y), z)(k^*(E)) = \Theta^{\cZ}  \left((\Theta^{\wt \cZ})^{-1} \bs^{\wt \cZ, \op{amb}}(\tau_{\wt \cZ}(\bm y), z)( \wt k^* \wt \pi_* \mathbb{FM} \ol j_* E)\right)\vline_{\tilde y_{r+1} = 0}.$$ \end{itemize}
\end{theorem}

\begin{proof}

The first bullet follows immediately from Theorem~\ref{t:nctc} and Theorem~\ref{t:qsd}, as does the statement about the pairing for $\wt \cZ$.  These theorems further show that the integral structure on $\bnab$ coincides with the integral structure of $\tau_{\wt \cZ}^* \nabla^{\cZ, \op{amb}}$. 

The second bullet follows immediately from Proposition~\ref{p:moninv} and Theorem~\ref{t:qsd}, as does the statements about the pairing for $\cZ$.  
These results also imply that the integral lattice of $\bnab|_{u_\bff H^*_{\op{CR}}(\ol \cT)}$ induced by $K^0_\cX(\cT)$, when restricted to $y_{r+1} = 0$, maps to the lattice of integral solutions of $\tau_{\cZ}^* \nabla^{\cZ, \op{amb}}$.

For the reader's convenience, we present the last statement of the theorem in detail:
\begin{align*}
&\Theta^{\cZ} \left((\Theta^{\wt \cZ})^{-1} \bs^{\wt \cZ, \op{amb}}(\tau_{\wt \cZ}(\bm y), z)( \wt k^* \wt  \pi_* \mathbb{FM} \ol j_* E)\right)\vline_{\tilde y_{r+1} = 0}\\  =&\Theta^{\cZ} \left( (\Theta^{\op{nar}})^{-1} \bs^{\wt \cT, \op{nar}}(\tau_{\wt \cT}(\bm y), z)(  \mathbb{FM} \ol j_* E) \right)\vline_{\tilde y_{r+1} = 0} \\
=&\bar \Delta_\cT \circ i^* \left(\bs^{\ol \cT, \op{nar}}(\tau_{\ol \cT}(\bm y), z)(\ol j_* E) \right)\vline_{\tilde y_{r+1} = 0} \\
=& \bar \Delta_\cT \left(\bs^{\cT, \op{nar}}(i^*\circ \tau_{\ol \cT}(\bm y), z)(j_* E)\right) \\
=& \bs^{\cZ, \op{amb}}(\tau_\cZ(\bm y), z)(k^*(E)).
\end{align*}
Here the first and last equality are Theorem~\ref{t:qsd}, the second equality is Theorem~\ref{t:nctc}, and the third is Proposition~\ref{p:moninv}.
\end{proof}

\subsection{Special cases}
In some cases a stronger statement is possible. Consider the following two conditions: \begin{enumerate}
\item \label{c1} The map $$i^*|_{u_\bff H^*_{\op{CR}}(\ol \cT)} : u_\bff H^*_{\op{CR}}(\ol \cT) \to  H^*_{\op{CR, nar}}(\cT)$$ is an isomorphism.
\item \label{c2} There is an equality $$u_\bff H^*_{\op{CR}}(\ol \cT) =  \bigoplus_{\nu \in \wh \KK_{\omega_-}^{\op{int}}/\wh \LL} \op{ker}\left( u_\bee \cdot - : H^*_{\op{nar}}(\ol \cT_\nu) \to H^*_{\op{nar}}(\ol \cT_\nu)\right).$$
  \end{enumerate}
When condition~\eqref{c1} holds, the narrow quantum $D$-module of $\cT$ may be identified with a submodule of the monodromy invariant part of  $\nabla^{\ol \cT, \op{nar}}$ rather than a subquotient. 
If condition~\ref{c2} holds,
then by Lemma~\ref{l:moninvt}  $\nabla^T$ may be identified with a quotient of the monodromy invariant part of $\nabla^{\ol \cT, \op{nar}}$ rather than a subquotient.
 If both of the above conditions are satisfied, then the map $\Theta^\cZ$ identifies $QDM_{\op{nar}}(\cT)$ with (the restriction to $\tilde y_{r+1} = 0$ of) the monodromy invariant part of $\nabla^{\ol \cT, \op{nar}}$.  This in turn implies a stronger  form of Theorem~\ref{t:mt}.

Given condition~\eqref{c1}, the map $i^*|_{u_\bff H^*_{\op{CR}}(\ol \cT)}$ may be inverted, which in turn defines a map
$\sigma: H^*_{\op{CR, nar}}(\cT) \to H^*_{\op{CR, nar}}(\ol \cT)$ given by composing $(i^*|_{u_\bff H^*_{\op{CR}}(\ol \cT)})^{-1}$ with the inclusion $u_\bff H^*_{\op{CR}}(\ol \cT) \subset H^*_{\op{CR, nar}}(\ol \cT)$.  

Under conditions~\eqref{c1} and~\eqref{c2}, Theorem~\ref{t:mt} may be rephrased as follows.
\begin{theorem}\label{t:mtstrong} When conditions~\eqref{c1} and~\eqref{c2} are satisfied, 
the quantum $D$-module $\bnab$ on $ \wh \cM^\circ$ satisfies the following:
\begin{itemize}
\item In a neighborhood of $y_{r+1}= 0$, $\bnab$ is gauge-equivalent to $\tau_{\wt \cZ}^* \nabla^{\wt \cZ, \op{amb}}$ via the transformation $\Theta^{\wt \cZ} =  \bar \Delta_{\wt \cT}\circ \Theta^{\op{nar}}$;
\item The monodromy invariant part of $\bnab$ around $\tilde y_{r+1} = 0$, when restricted to $ \tilde y_{r+1} = 0$, is guage equivalent to 
 to $\tau_{\cZ}^* \nabla^{\cZ, \op{amb}}$ via the transformation $\Theta^\cZ=  \bar \Delta_\cT \circ i^*$;
\item For $\alpha, \beta \in H^*_{\op{CR, amb}}( \cZ)$,
$$S^{\wt \cZ, \op{amb}}(\Theta^{\wt \cZ} \circ (\Theta^{\cZ})^{-1} \alpha, \Theta^{\wt \cZ} \circ (\Theta^{\cZ})^{-1} \beta) = S^{\cZ, \op{amb}}(\alpha, \beta),$$
where $(\Theta^{\cZ})^{-1} := \sigma \circ (\bar \Delta_{\wt \cT})^{-1}.$
\item  
The integral lattice of $\tau_{\cZ}^* \nabla^{\cZ, \op{amb}}$ is equal to the restriction of the monodromy invariant sublattice of the integral lattice of $\tau_{\wt \cZ}^* \nabla^{\wt \cZ, \op{amb}}$ to $\tilde y_{r+1} = 0$.
 \end{itemize}
\end{theorem}

We conclude with examples illustrating when the above conditions hold, and give examples showing they do not always hold.
\begin{example} 
Consider Example~\ref{e:proj}, where $\cZ$ is a degree-$d$ hypersurface in projective space $\cX = \PP^{m-1}$, and $\wt \cZ$ is a hypersurface in $\wt \cX = \op{Bl}_{\PP^{m-k-1}} \PP^{m-1}$.  The GIT description of $\ol \cT$ is given by $\wh D_1 = \cdots = \wh D_k = (1,1)$, $\wh D_{k+1} = \cdots = \wh D_m = (1, 0)$, $\wh D_\bee = (0, -1)$, and $\wh D_\bff = (-d, 1-k)$, with stability condition $\omega_- = (1, -\epsilon)$.    We check below whether~\eqref{c1} and~\eqref{c2} are satisfied.  

Note that both conditions concern only the cohomology of twisted sectors $\ol \cT_\nu$ for $\nu \in \wh \KK_{\omega_-}^{\op{int}}/\wh \LL$.  In this case that is only the untwisted sector.  Thus we restrict our attention to $\ol \cT$.
By \eqref{e:SR}, and the description of the fan given in Proposition~\ref{p:ott}, one computes
\begin{equation}\label{e:cohtbar}H^*(\ol \cT) = \CC[u, e] / \langle eu^{m-k}, u^m, (-d u + (k-1)e)e \rangle,\end{equation}
where $u$ and $e$ denote the divisors in $H^2(\wt \cT)$ corresponding to $D_m$ and $D_\bee$ respectively.
A homogeneous basis is given by 
$$ \{ 1, u, \ldots, u^m, e, eu, \ldots, eu^{m-k-1}\}.$$
On the other hand, $\cT = \cc O_{\PP^{m-1}}(-d)$ and 
$$H^*(\cT) = H^*(\PP^{m-1}) = \CC[u]/\langle u^m \rangle.$$
The map $i^*: H^*(\ol \cT) \to H^*(\cT)$ simply sends $u$ to $u$ and $e$ to $0$.  The vector space $u_\bff H^*(\ol \cT)$ is equal to $ \langle -d u + (k-1)e\rangle$.  It follows immediately that condition~\eqref{c1} is satisfied for all choices of $m, k, d$.

By Lemma~\ref{l:narcoh}, the narrow cohomology is generated as a module by  $u_\bff$ and $u_1  \cdots  u_k$.  In this case the module is $\langle -d u + (k-1)e, (u-e)^{k}\rangle.$   By Condition~\ref{a:as1} for $\wt D$, $d \geq k-1$.  If $d = k-1$, then $(-d u + (k-1)e)u^{k-1}$ and $(u-e)^{k}$ are scalar multiples of each other, and thus $H^*_{\op{nar}}(\ol \cT) = \langle -d u + (k-1)e\rangle$.  Condition~\eqref{c2} will therefore be satisfied automatically.

We next consider the case $d > k-1$.  In this case one checks that 
$u^k$ lies in the span of  $(-d u + (k-1)e)u^{k-1}$ and $(u-e)^{k}$, and therefore lies in $H^*_{\op{nar}}(\ol \cT)$.  
We can express the narrow cohomology more simply as 
$$H^*_{\op{nar}}(\ol \cT) = \langle -d u + (k-1)e, u^{k}\rangle.$$
From \eqref{e:cohtbar}, 
$\op{ker}\left( u_\bee \cdot - = e \cdot - : H^*(\ol \cT) \to H^*(\ol \cT)\right)$ is equal to 
$\langle -d u + (k-1)e, u^{m-k}\rangle.$  This module is rank one in degrees $$1, 2, \ldots, m-k-1, m-k+1, \ldots, m-1$$ and rank two in degree $m-k$, with generators $(-d u + (k-1)e)u^{m-k-1}$ and $u^{m-k}$.  We conclude that 
$$\op{ker}\left( u_\bee \cdot - = e \cdot - : H^*(\ol \cT) \to H^*(\ol \cT)\right) = u_\bff H^*(\ol \cT) \oplus \CC \cdot u^{m-k}.$$
Condition~\eqref{c2} is therefore satisfied if and only if $u^{m-k} \notin H^*_{\op{nar}}(\ol \cT)$, which in turn holds if and only if $m-k < k$.  

In summary, in the setting of Example~\ref{e:proj}, condition~\eqref{c1} is always satisfied.  Condition~\eqref{c2} is satisfied if and only if either $d = k-1$ or if $m-k<k$.  In particular, the example of the cubic transition of the quintic 3-fold ($m=d=5$ and $k = 4$), both conditions are satisfied and Theorem~\ref{t:mtstrong}.  This is consistent with \cite{MiCub}.

\end{example}

\begin{example}
For an example where condition~\eqref{c1} fails, we consider a particular case of Example~\ref{e:wp}.  Let $m = 5$, $k = 2$, $D_1 = D_2 = D_3 = D_6 = 1$, $D_4 = D_5 = 2$, and $\omega = 1$.  Then $\cX = \PP(1,1,1,2,2,1)$.  Let $a_1 =  \cdots = a_4 = a_5 = 0$ and $a_6 = 8$.  Then $\cZ$ is a degree $8$ Calabi--Yau hypersurface in $\cX$ and $\wt \cZ$ is a hypersurface in $\op{BL}_{\PP(1,2,2,1)} \PP(1,1,1,2,2,1)$.
Note that Condition~\ref{a:as1} holds for both $D$ and $\wt D$ in this case.

The inertia stack $I\cX$ has a twisted sector $\cX_{1/2} \cong \PP(2,2)$ corresponding to the element $1/2 \in \KK_\omega /\LL$.  The corresponding twisted sectors in $I\cT$ and $I \ol \cT$ are 
\begin{align*}
\cT_{1/2} &=  \left(\cc O_{\PP(2,2)}(-8)\right) \\
\ol \cT_{(1/2, 0)} &= \PP\left(\cc O_{\PP(2,2)} \oplus \cc O_{\PP(2,2)}(-8)\right).
\end{align*} 
We compute the cohomology rings
\begin{align*}
H^*(\cT_{1/2}) & = \CC[u]/\langle u^2 \rangle \\
H^*(\ol \cT_{(1/2, 0)}) &= \CC[u, e]/\langle u^2, (-8u + e)e\rangle. \end{align*}
Again the map $i^*_{1/2}: H^*(\ol \cT_{(1/2, 0)}) \to H^*(\cT_{1/2})$ is given by sending $u$ to $u$ and $e$ to $0$.  Note that the element $u_\bff \cdot u = (-8u + e)u \in H^*(\ol \cT_{(1/2, 0)})$ is a nonzero element in $\op{ker}(i^*|_{u_\bff H^*_{\op{CR}}(\ol \cT)})$.  Thus condition~\ref{c1} fails.  
\end{example}
\begin{remark}
The general statement of Theorem~\ref{t:mt}, relating $QDM_{\op{amb}}(\cZ)$ to a subquotient of $QDM_{\op{amb}}(\wt \cZ)$ (rather than simply a sub-module) is not surprising.  Indeed the examples of extremal correspondences in \cite{IX} take a similar form.
\end{remark}
It would be interesting to better understand what conditions on $\cZ$ and $\wt \cZ$ would imply conditions \eqref{c1} and \eqref{c2} above.

\bibliographystyle{plain}
\bibliography{references}

\end{document}